\documentclass[11pt]{article}

\usepackage{epsfig}
\usepackage{amssymb, latexsym}
\usepackage{amscd}
\usepackage[all]{xy}
\usepackage{epsf}
\usepackage[mathscr]{eucal}

\usepackage{amsmath,amsfonts,amscd,amssymb,amsthm}

\usepackage[titletoc]{appendix}
\usepackage[sc]{titlesec}

\long\def\comment#1\endcomment{}

\theoremstyle{plain}

\newtheorem{theorem}{\sc Theorem}[section]
\newtheorem{lemma}[theorem]{\sc Lemma}

\newtheorem{prop}[theorem]{\sc Proposition}

\theoremstyle{plain}

\newtheorem{defn}[theorem]{\sc Definition}

\theoremstyle{exercise}

\newtheorem{example}[theorem]{\sc Example}

\makeatletter \@addtoreset{equation}{section} \makeatother

\def\eqref#1{\thetag{\ref{#1}}}

\let\latexref=\ref
\def\ref#1{{\normalfont{\latexref{#1}}}}

\newcommand{\ldot}{{\:\raisebox{2,3pt}{\text{\circle*{1.5}}}}}
\newcommand{\udot}{{\:\raisebox{3pt}{\text{\circle*{1.5}}}}}
\def\dlim_#1{{\displaystyle\lim_{#1}}^\hdot}

\newcommand{\Ker}{\operatorname{{\rm Ker}}}

\newcommand{\id}{\operatorname{\rm id}}

\newcommand{\Ob}{\mathrm{Ob}}

\renewcommand{\Bar}{\mathrm{Bar}}
\newcommand{\Cobar}{\mathrm{Cobar}}

\newcommand{\Hom}{\mathrm{Hom}}

\newcommand{\dg}{{dg}}
\newcommand{\Ho}{\mathrm{Ho}}
\newcommand{\Cone}{\mathrm{Cone}}

\newcommand{\Id}{\mathrm{Id}}

\newcommand{\pretr}{\mathrm{pre-tr}}

\newcommand{\Vect}{\mathscr{V}ect}

\newcommand{\Cat}{{\mathscr{C}at}}

\renewcommand{\k}{\Bbbk}

\newcommand{\Assoc}{{{\mathcal{A}ssoc}}}

\newcommand{\pprime}{{\prime\prime}}

\newcommand{\Eq}{\mathrm{Eq}}
\newcommand{\Coeq}{\mathrm{Coeq}}
\newcommand{\Graphs}{{\mathscr{G}raphs}}
\newcommand{\dgwu}{{dgwu}}
\newcommand{\dgu}{{dg}}
\newcommand{\Fib}{\mathrm{Fib}}
\newcommand{\Surj}{\mathrm{Surj}}
\newcommand{\Fin}{{\mathcal{F}in}}
\newcommand{\inj}{\mathrm{inj}}
\newcommand{\cell}{\mathrm{cell}}
\newcommand{\cof}{\mathrm{cof}}
\renewcommand{\u}{u}

\newcommand{\sevafigc}[4]{\begin{figure}[h]\centerline{
 \epsfig{file=#1,width=#2,angle=#3}}
\bigskip\caption{#4}\end{figure}}

\title{\sc{A Quillen model structure on the category of Kontsevich-Soibelman weakly unital dg categories}}

\author{\sc{Piergiorgio Panero and Boris Shoikhet}}
\date{}

\begin{document}\maketitle
{\footnotesize
\begin{center}{\parbox{4,5in}{{\sc Abstract.}
In this paper, we study {\it weakly unital} dg categories as they were defined  by Kontsevich and Soibelman [KS, Sect.4]. 
We construct a cofibrantly generated Quillen model structure on the category $\Cat_{\dgwu}(\k)$ of small weakly unital dg categories over a field $\k$. Our model structure can be thought of as an extension of the model structure on the category $\Cat_\dg(\k)$ of (strictly unital) small dg categories over $\k$, due to Tabuada [Tab]. More precisely, we show that the imbedding of $\Cat_\dg(\k)$ to $\Cat_{dgwu}(\k)$ is a right adjoint of a Quillen pair of functors. We prove that this Quillen pair is, in turn, a Quillen equivalence. In course of the proof, we study a non-symmetric dg operad $\mathcal{O}$, governing the weakly unital dg categories, which is encoded in the Kontsevich-Soibelman definition. We prove that this dg operad is quasi-isomorphic to the operad $\Assoc_+$ of unital associative algebras.

}}
\end{center}
}
\section*{\sc Introduction}
\subsection{\sc}
Weakly unital $A_\infty$ categories firstly appeared in the definition of Fukaya category in Homological mirror symmetry [K2]. 
Since that, weakly unital dg and $A_\infty$ categories have been studied by many authors, e.g. [LyMa], [Ly], [LH], [KS], [COS] among the others. 
Currently there are known three different definitions of a weakly unital $A_\infty$ (or dg) category [LyMa]. These three definitions are due to Fukaya, to Lyubashenko, and to Kontsevich-Soibelman, correspondingly. It was proven loc.cit. that the three definitions are equivalent, which means that if a given $A_\infty$ category is weakly unital in one sense it is also weakly unital in another. Nevertheless, the three categories of weakly unital $A_\infty$ categories are not equivalent. Their homotopy categories were supposed to be equivalent, and equivalent to the homotopy category of strictly unital dg categories. Our Theorem \ref{theorq2} confirms this claim for the Kontsevich-Soibelman definition.\footnote{Theorem 2.2 of [COS] confirms this claim for the Lyubashenko definition. The Lyubashenko weakly unital dg categories seemingly do not admit a closed model structure, and the proof in loc.cit. is direct. }.

The Kontsevich-Soibelman definition is, in authors' opinion, the most manageable. If one restricts to dg categories, the category $\Cat_\dgwu(\k)$ of small Kontsevich-Soibelman weakly unital dg categories over a field $\k$ admits all small limits and colimits (Theorem \ref{cococo}). Our main results show that there is a closed model structure on $\Cat_\dgwu(\k)$, extending the Tabuada closed model structure [Tab] on the category $\Cat_\dg(\k)$ of small unital dg categories over $\k$, and that the two closed model categories $\Cat_\dg(\k)$ and $\Cat_\dgwu(\k)$ are Quillen equivalent (Theorem \ref{theoremmain1} and Theorem \ref{theorq2}). 
\subsection{\sc }
Weakly unital dg categories emerge as well in some elementary algebraic constructions. Thus, let $A$ be a strictly unital dg algebra over $\k$. 
Then its bar-cobar resolution $\Cobar(\Bar(A))$ is a very nice ``cofibrant resolution'' of $A$. It is only true if it is considered as a non-unital dg algebra, because $\Cobar(\Bar(A))$ lacks a strict unit. In fact, $\Cobar(\Bar(A))$ is Kontsevich-Soibelman weakly unital, see Example \ref{barcobar1}. 

On the other hand, the bar-cobar resolution is a very natural resolution and one likes to consider it as a cofibrant replacement of $A$, when one computes $\Hom$ sets in the homotopy category. Certainly, $\Hom(\Cobar(\Bar(A)),B)$ in the non-unital setting is the set of {\it all} $A_\infty$ maps (or $A_\infty$ functors, for the case of dg categories). However, it is well-known [LH] that the correct $\Hom$ set in the homotopy category is defined via the {\it unital} $A_\infty$ maps (corresp., unital $A_\infty$ functors).\footnote{Recall that an $A_\infty$ map $F\colon A\to B$ is {\it unital} if $F_1(1_A)=1_B$ and $F_k(\dots,1_A,\dots)=0$ for $k\ge 2$.} The reason is that one has to take $\Hom(\Cobar(\Bar(A)),B)$ in the category of (Kontsevich-Soibelman) {\it weakly unital} dg categories, see Definition \ref{defwu}, and it gives rise exactly to the unital $A_\infty$ functors $A\to B$, see Example \ref{barcobar2}. 

One of our goals is to develop a suitable categorical environment in which the mentioned facts fit naturally. Some other applications will appear in our next paper. 

\subsection{\sc }
Let us outline in more detail our main results and the organization of the paper. 

In Section 1, we recall the Kontsevich-Soibelman definition of weakly unital dg categories and of their morphisms, which gives rise to a category $\Cat_\dgwu(\k)$. After that, we prove that the category $\Cat_\dgwu(\k)$ admits all small limits and colimits. The products, the coproducts, and the equalizers are constructed directly. The coequalizers are less trivial, to define them we use technique of monads. We adapt some ideas of [Wo] and [Li], where enriched strictly unital case is treated. We construct a monad $T$ on the category of dg graphs and prove in Theorem \ref{wumonadic} that the categories of $T$-algebras and of weakly unital dg categories are equivalent. 
The coequalizers are constructed in Proposition \ref{propcoeqmain}. We also construct a non-symmetric dg operad $\mathcal{O}$ such that $\mathcal{O}$-algebras in dg graphs are exactly weakly unital dg categories.

In Section 2, we prove Theorem \ref{theoremmain1} which says that there is a cofibrantly generated closed model structure on $\Cat_\dgwu(\k)$. We construct sets of generating cofibrations $I$ and of generating acyclic cofibrations $J$ which are paralleled to those in [Tab]. There is a trick, employed in Lemma \ref{klemma}, with  the acyclic cofibration $\mathscr{A}\to \mathscr{K}$ where $\mathscr{K}$ is the Kontsevich dg category with two objects. Namely, we notice that, for any closed degree 0 morphism $\xi$ in a weakly unital dg category $C$, the replacement of $\xi$ by $\xi^\prime=1\cdot \xi\cdot 1$ does not affect the class $[\xi]\in H^0(C)$, and,
at the same time, $\xi^\prime$ satisfies $1\cdot \xi^\prime=\xi^\prime\cdot 1=\xi^\prime$.  It makes us possible to use Tabuada's acyclic cofibration $\mathscr{A}\to\mathscr{K}$ in the weakly unital case, without any adjustment. Another new and subtle place is Lemma \ref{lemmafibj}, which, even in the unital case, simplifies the argument. In the weakly unital case it provides, seemingly, the only possible way to prove Theorem \ref{theoremmain1}.

In Section 3, we provide an adjoint pair of functors 
$$
L\colon \Cat_\dgwu(\k)\rightleftarrows \Cat_\dg(\k): R
$$
and prove, in Proposition \ref{propqpair}, that it is a Quillen pair. Moreover, we show  in Theorem \ref{theorq2} that it is a Quillen equivalence, if the natural projection of dg operads $\mathcal{O}\to\Assoc_+$ is a quasi-isomorphism.

Finally, in Section 4 we prove Theorem \ref{theoremmop} which states that the natural projection $p\colon \mathcal{O}\to\Assoc_+$ is a quasi-isomorphism of dg operads. It completes the proof of Theorem \ref{theorq2}. The proof of Theorem \ref{theoremmop} goes by a quite tricky computation with spectral sequences. 

In Appendix A, we provide some detail to the proof of [Dr, Lemma 3.7], which we employ in the proof of Lemma \ref{lemmafibj}.

\subsection*{\sc}
\subsubsection*{\sc Acknowledgements}
We are thankful to Bernhard Keller for his encouragement. 
The work of both authors was partially supported by the FWO Research Project Nr. G060118N.
The work of B.Sh. was partially supported by the Russian Academic Excellence Project ‘5-100’.

\section{\sc Weakly unital dg categories}
\subsection{\sc The definition}
We adapt the definition of {\it weakly unital dg categories} given in [KS, Sect. 4], where a more general context of $A_\infty$ categories is considered.

\subsubsection{\sc }

Let $A$ be a (non-unital) dg category. Denote by $\k_A$ the unital dg category whose objects are $Ob(A)$, for any $X\in\Ob(A)$ $\k_A(X,X)=\k$, and $\k_A(X,Y)=0$ for $X\ne Y$. We denote by $1_X$ the unit element in $\k_A(X,X)$. By abuse of notations, we denote, for a non-unital dg category $A$, by $A\oplus \k_A$ the unital dg category having the same objects that $A$, and
$$
(A\oplus\k_A)(X,Y)=\begin{cases}
A(X,Y)&X\ne Y\\
A(X,X)\oplus \k_X&X=Y
\end{cases}
$$
One has a natural imbedding $i\colon A\to A\oplus \k_A$ sending $X$ to $X$, and $f\in A(X,X)$ to the pair $(f,0)\in (A\oplus\k_A)(X,X)$.

\begin{defn}\label{defwu}{\rm
A {\it weakly unital dg category} $A$ over $\k$ is a non-unital dg category $A$ over $\k$, with a distignuished element $\id_X\in A(X,X)^0$, for any object $X$ in $A$, such that $d(\id_X)=0$ and $\id_X\circ \id_X=\id_X$, subject to the following condition.
One requires that there exists an $A_\infty$ functor $p\colon A\oplus \k_A\to A$, which is the identity map on the objects, such that $p\circ i=\id_A$, and which fulfils the conditions:
$$
p_1(1_X)=\id_X, \ p_n(1_X,\dots,1_X)=0 \text{\ for $n\ge 2$}, \text{ for any }X\in\Ob(A)
$$
 
}
\end{defn}
\begin{example}\label{uwu}{\rm
Let $A$ be a strictly unital dg category. Define $p\colon A\oplus\k_A\to A$ as $p_1|_{A(X,Y)}=\id$, $p(1_X)=\id_X$, $p_n=0$ for $n\ge 2$. Then $p$ is a dg functor, and $p\circ i=\id$. It makes a strictly unital dg category a weakly unital dg category. 
}
\end{example}

\begin{lemma}\label{lemmatriv1}
Let $A$ be a weakly unital dg category. Then the homotopy category $H^0(A)$ is a strictly unital dg category. 
\end{lemma}   
\begin{proof}
The map $[p_1]\colon H^0(A)\oplus \k_{H^0(A)}\to H^0(A)$, induced by the first Taylor component $p_1$ of the $A_\infty$ functor $p$, is a dg algebra map. One has $[p](1_X)=\id_X$ and $[p]\circ [i]=\id$. It follows that $\id_X\circ f=f\circ \id_X=f$, for any $f\in H^0(A)$. 
\end{proof}

\begin{example}\label{barcobar1}{\rm
Let $A$ be an associative dg algebra over $\k$, with a strict unit $1_A$. Consider $C=\Cobar_+(\Bar_+(A))$ where $\Bar_+(A)$ is the bar-complex of $A$, which is non-counital dg coalgebra (thus, $\Bar_+(A)=T(A[1])/\k$ as a graded space), and $\Cobar_+(B)$ is the non-unital dg algebra (as a graded space, $\Cobar_+(B)=T(B[-1])/\k$). It is well-known that the natural projection $\Cobar_+(\Bar_+(A))\to A$ is a quasi-isomorphism {\it of non-unital dg algebras}. We claim that $\Cobar_+(\Bar_+(A))$ is (almost) weakly unital, whose weak unit is $1_A\in \Cobar_+(\Bar_+(A))$. By ``almost'' we mean that for $p_n$ defined below it is not true that $p_n(1,1,\dots,1)=0$ for $n\ge 2$. (One can easily take a quotient by the corresponding acyclic ideal, or alternatively one can regard it as an object of the category $\Cat^\prime_\dgwu(\k)$ rather than an object of $\Cat_\dgwu(\k)$, see Section \ref{section112}). 

We use notations $\omega=a_1\otimes\dots\otimes a_\ell\in \Bar_+(A)$ for monomial bar-chains, and $c=\omega_1\boxtimes \omega_2\boxtimes\dots\boxtimes\omega_k$ for monomial elements in $\Cobar_+(\Bar_+(A))$. 

Define $p_n(x_1,\dots,x_n)$, where each $x_i$ is either $1$ or a monomial $c\in\Cobar_+(\Bar_+(A))$, as follows. 

(1): We set $p_n(x_1,\dots,x_n)$ to be 0 if for some $1\le i\le n-1$ both $x_i,x_{i+1}$ are elements in $\Cobar_+(\Bar_+(A))$. 
(2): Otherwise, let $x_i,\dots, x_{i+j+1}$ be a fragment of the sequence $x_1,\dots, x_n$ such that $x_i=\omega_1\boxtimes\dots\boxtimes \omega_a\in\Cobar_+(\Bar_+(A))$, $x_{i+1}=\dots=x_{i+j}=1$, $x_{i+j+1}=\omega^\prime_1\boxtimes\dots\boxtimes\omega_b^\prime\in\Cobar_+(\Bar_+(A))$. Then we replace the fragment $x_i,x_{i+1},\dots,x_{i+j+1}$ by the following element $
\gamma$ in $\Cobar_+(\Bar_+(A))$: 
$$
\gamma=\omega_1\boxtimes\dots\boxtimes \omega_{a-1}\boxtimes (\omega_a\otimes\underset{\text{$j$ factors $\id$}}{ \id\otimes \dots\otimes \id}\otimes \omega_1^\prime)\boxtimes \dots\boxtimes \omega_b^\prime
$$
(3): We perform such replacements succesively for all suitable fragment, and finally we get an element in $\Cobar_+(\Bar_+(A))$, of degree $\sum \deg x_i-n+1$. By definition, this element is $p_n(x_1,\dots,x_n)$. By a suitable fragment we mean either the case  considered above, when a group of succesive 1s is surrounded by elements of $\Cobar_+(\Bar_+(A))$ from both sides, or one of the two extreme case: if $x_1=1$, the leftmost $1,1,\dots,1,x_i$ is a suitable fragment, and similarly if $x_n=1$, the rightmost fragment $x_s,1,\dots,1$ is also suitable.

One easily checks that the constructed $\{p_n\}_{n\ge 1}$ defines an $A_\infty$ morphism $p\colon \Cobar_+(\Bar_+(A))\oplus \k 1\to \Cobar_+(\Bar_+(A))$ such that $p\circ i=\id$. 

The construction for the case of $\Cobar_+(\Bar_+(C))$, for $C$ a dg category, is similar. 
}
\end{example}

\subsubsection{\sc }\label{section112}
We endow the weakly unital dg categories with a category structure, as follows.

\begin{defn}\label{defwucat}{\rm
Let $C,D$ be weakly unital dg categories, denote by $i^C\colon C\to C\oplus\k_{C}$, $i^D\colon D\to D\oplus \k_D$ and by $p^C\colon C\oplus \k_{C}\to C$, $p^D\colon D\oplus\k_D\to D$ the corresponding functors (see Definition \ref{defwu}). {\it A weakly unital dg functor} $F\colon C\to D$ is defined as a dg functor of non-unital dg categories 
such that the diagram below commutes:
\begin{equation}\label{mapwu}
\xymatrix{
C\oplus\k_{C}\ar[r]^{F\oplus\id}\ar[d]_{p^C}&D\oplus\k_{D}\ar[d]^{p^D}\\
C\ar[r]^{F}&D
}
\end{equation}
Note that the upper horizontal map $F\oplus\id$ is automatically a dg functor of unital dg categories, and $p_1,p_2$ are $A_\infty$ maps.
Note that it follows that
\begin{equation}\label{eq9}
F(\id_X)=\id_{F(X)}
\end{equation}
for any $X\in\Ob(C)$.
}
\end{defn}
Denote by $\Cat_{dgwu}(\k)$ the category of small weakly unital dg categories over $\k$. 

Similarly we define a category $\Cat_\dgwu^\prime(\k)$. Its objects are defined as the objects of $\Cat_{\dgwu}(\k)$ but with dropped conditions $p_n(1,\dots,1)=0$ for $n\ge 2$ and $p(1)\cdot p(1)=p(1)$. The morphisms are defined as for the category $\Cat_\dgwu(\k)$. One sees that the weakly unital dg algebra $\Cobar_+(\Bar_+(A))$, constructed in Example \ref{barcobar1}, is an object of $\Cat^\prime_\dgwu(\k)$ (but is not an object of $\Cat_\dgwu(\k)$). 

Note that the commutativity of diagram \eqref{mapwu} implies
\begin{equation}\label{eq10}
F(p_n^C(f_1\otimes\dots\otimes f_n))=p_n^D(F(f_1)\otimes\dots\otimes F(f_n))
\end{equation}
for any $n$ morphisms $f_1,\dots,f_n$ in $C$.

\begin{lemma}\label{lemmatriv2}
Let $F\colon C\to D$ be a weakly unital dg functor between weakly unital dg categories. Then it defines a $\k$-linear functor 
$H^0(F)\colon H^0(C)\to H^0(D)$ of unital $\k$-linear categories. 
\end{lemma}
It is clear.

\qed

\begin{example}\label{barcobar2}{\rm
Let $A$ be a strictly unital dg algebra, consider the weakly unital dg algebra $C=\Cobar_+(\Bar_+(A))$ (which belongs to $\Cat_\dgwu^\prime(\k))$), constructed in Example \ref{barcobar1}. Let $D$ be a strictly unital dg algebra. Then the set $\Hom_{\Cat_\dgwu^\prime(\k)}(C,D)$ is identified with the set of {\it unital} $A_\infty$ maps $A\to D$. (Recall that for strictly unital dg algebras $A,D$, an $A_\infty$ morphism $f\colon A\to D$ map is called {\it unital} if $f_1(1_A)=1_D$, and $f_n(a_1,\dots,a_n)=0$ if $n\ge 2$ and at least one argument $a_i=1_A$).

One has a similar description for the case of dg categories. 
}
\end{example}
\subsubsection{\sc The small (co)limits in $\Cat_{dgwu}(\k)$}\label{limcolimfirst}
It is true that the dg category $\Cat_{dgwu}(\k)$ is small complete and small cocomplete. 
One constructs directly small products and small coproducts. The equalizers are also straightforward, as follows.

Let $F,G\colon C\to D$ be two morphisms. Define $\Eq(F,G)$ as the dg category whose objects are 
$$
\Ob(\Eq(F,G))=\{X\in \Ob(C)| F(X)=G(X)\}
$$
Let $X,Y\in \Ob(\Eq(F,G))$. Define
$$
\Eq(F,G)(X,Y)=\{f\in C(X,Y)| F(f)=G(f)\}
$$
It is clear that $\Eq(F,G)$ is a non-unital dg category. For any $X\in \Ob(\Eq(F,G))$, $F(\id_X)=\id_{F(X)}$ and $G(\id_X)=\id_{G(X)}$, therefore $\id_X\in \Eq(F,G)(X,X)$.

One has to construct an $A_\infty$ functor $p\colon \Eq(F,G)\oplus\k_{\Eq(F,G)}\to\Eq(F,G)$ such that $p_1(1_X)=\id_X$, and $p\circ i=\id$. We define 
$$
p_n^{\Eq(F,G)}(f_1\otimes\dots\otimes f_n)=p_n^C(f_1\otimes\dots\otimes f_n)
$$
One has to check that $p_n^{\Eq(F,G)}(f_1\otimes\dots\otimes f_n)$ is a morphism in $\Eq(F,G)$, that is,
\begin{equation}\label{eq11}
F(p_n^{C}(f_1\otimes\dots\otimes f_n))=G(p_n^{C}(f_1\otimes\dots\otimes f_n))
\end{equation}
From \eqref{eq10} one gets
$$
F(p_n^C(f_1\otimes\dots\otimes f_n))=p_n^D(F(f_1)\otimes\dots F(f_n))
$$
and
$$
G(p_n^C(f_1\otimes\dots\otimes f_n))=p_n^D(G(f_1)\otimes\dots\otimes G(f_n))
$$
Now \eqref{eq11} follows from $F(f_i)=G(f_i)$ for all $f_i$, which holds because all $f_i$ are morphisms in $\Eq(F,G)$. Thus, $\Eq(F,G)$ is a weakly unital dg category.

\smallskip

To construct the coequalizers is a harder task. For the category $\mathscr{V}{-}\Cat$ of small $\mathscr{V}$-enriched categories, the coequalizers were constructed in [Li] and [Wo], assuming $\mathscr{V}$ to be a symmetric monoidal closed and cocomplete, and were constructed in [BCSW] and [KL] in weaker assumptions on $\mathscr{V}$. 
All these proofs rely on the theory of monads. We associate a monad which governs the weakly unital dg categories in Section \ref{subsectionmonadwu}. 

We adapt the approach of [Wo] for a proof of existence of the coequalizers in $\Cat_{dgwu}(\k)$. We also prove the corresponding monadicity theorem.

\subsection{\sc The monad of weakly unital dg categories}

\subsubsection{\sc Reminder on monads}
Here we recall definions and some general facts on monads and algebras over monads. The reader is referred to [ML], [R] for more detail.

Let $\mathscr{C}$ be a category. Recall that a monad in $\mathscr{C}$ is given by an endofunctor $$T\colon \mathscr{C}\to\mathscr{C}$$ 
and natural transformations 
$$
\eta\colon \Id\Rightarrow T\text{  and  }\mu\colon T^2\Rightarrow T
$$
so that the following diagrams commute:
$$
\xymatrix{
T^3\ar@2[r]^{T\mu}\ar@2[d]_{\mu T}&T^2\ar@2[d]^{\mu}\\
T^2\ar@2[r]^{\mu}&T
}\hspace{20mm}
\xymatrix{
T\ar@2[r]^{\eta T}\ar@2[rd]_{\Id} & T^2\ar@2[d]^{\mu} & T\ar@2[l]_{T\eta}\ar@2[dl]^{\Id}\\
&T
}
$$
A monad appears from a pair of adjoint functors. Assume we have an adjoint pair
\begin{equation}\label{mon1}
F\colon \mathscr{C}\rightleftarrows\mathscr{D}: U
\end{equation}
with adjunction unit and counit $\eta\colon \Id_{\mathscr{C}}\Rightarrow UF$ and $\varepsilon\colon FU\Rightarrow\Id_{\mathscr{D}}$.

It gives rise to a monad in $\mathscr{C}$, defined as:
$$
T=UF,\ \ \eta=\eta\colon \Id_{\mathscr{C}}\Rightarrow T,\ \ \mu=U\epsilon F\colon T^2\Rightarrow T
$$

An {\it algebra $A$ over a monad $T$} is given by an object $A\in\mathscr{C}$ equipped with a morphism $a\colon TA\to A$ such that the following diagrams commute:
$$
\xymatrix{
A\ar[r]^{\eta_A}\ar[rd]_{\Id_A}&TA\ar[d]^{a}\\
&A}\hspace{20mm}
\xymatrix{
T^2A\ar[r]^{\mu_A}\ar[d]_{Ta}&TA\ar[d]^{a}\\
TA\ar[r]^{a}&A
}
$$
The morphisms of algebras over a monad $T$ are defined as morphisms $f\colon A\to B$ in $\mathscr{C}$ such that the natural diagram commutes.

The category of $T$-algebras is denotes by $\mathscr{C}^T$.

There is an adjunction 
$$
F^T\colon \mathscr{C}\rightleftarrows \mathscr{C}^T\colon U^T
$$
which by its own gives rise to a monad. 

There is a functor $\Phi\colon\mathscr{D}\to\mathscr{C}^T$, sending an object $Y$ of $\mathscr{D}$ to the $T$-algebra $A=UY$, with 
$a\colon TA=UFUY\to UY=A$ equal to $U\varepsilon_Y$. The functor $\Phi$ is called the {\it Eilenberg-Moore comparison functor}.

An adjunction \eqref{mon1} is called {\it monadic} if the functor $\Phi\colon \mathscr{D}\to\mathscr{C}^T$ is an equivalence. 

There is a criterium when an adjunction is monadic, called the {\it Beck monadicity theorem}. We recall its statement below. 

Recall that a {\it split coequalizer} in a category is a diagram 
\vspace{3mm}
$$
\xymatrix{
A\ar@<0.5ex>[r]^{f}\ar@<-0.5ex>[r]_{g}&B\ar[r]^{h}\ar@/_1.5pc/[l]_s&C\ar@/_1.5pc/[l]_t
}
$$
such that 
\begin{itemize}
\item[(1)] $f\circ s=\id_B$,
\item[(2)] $g\circ s=t\circ h$,
\item[(3)] $h\circ t=\id_C$,
\item[(4)] $h\circ f=h\circ g$
\end{itemize}

Recall
\begin{lemma}\label{lemmaac}
A split coequalizer is a coequalizer, and is an absolute coequalizer (that is, is preserved by any functor).
\end{lemma}
It is enough to prove the first statement, because a split equalizer remains a split equalizer after application of any functor. See e.g. [R, Lemma 5.4.6] for detail.

\qed

Given a pair $$A\overset{f}{\underset{g}{\rightrightarrows}}B$$
in a category $\mathscr{D}$, and a functor $U\colon \mathscr{D}\to\mathscr{C}$, we say that this pair is {\it $U$-split} if the pair
$$
U(A)\overset{f}{\underset{g}{\rightrightarrows}}U(B)$$ in $\mathscr{C}$ can be extended to a split coequalizer.

\begin{theorem}\label{montheorem1}
Let $F\colon \mathscr{C}\rightleftarrows\mathscr{D}\colon U$ be a pair of adjoint functors, and let $T=UF$ be the corresponding monad. Consider the Eilenberg-MacLane comparison functor $\Phi\colon \mathscr{D}\to\mathscr{C}^T$. Then:
\begin{itemize}
\item[(1)] if $\mathscr{D}$ has coequalizers of all $U$-split pairs, the functor $\Phi$ has a left adjoint $\Psi\colon \mathscr{C}^T\to\mathscr{D}$,
\item[(2)] if, furthermore, $U$ preserves coequalizers of all $U$-split pairs, the unit $\Id_{\mathscr{C}^T}\Rightarrow \Phi\Psi$ is an isomorphism,
\item[(3)] if, furthermore, $U$ reflects isomorphisms (that is, $U(f)$ an isomorphism implies $f$ an isomorphism), the counit $\Psi\Phi\Rightarrow \Id_{D}$ is also an isomorphism.
\end{itemize}
Therefore, if (1)-(3) hold, $(U,F)$ is monadic. Conversely, if $(U,F)$ is monadic, conditions (1)-(3) hold.
\end{theorem}
The reader is referred to [ML] or  [R] for a proof.

\subsubsection{\sc Reminder on monads, II}
There is another monadicity theorem, which gives sufficient but not necessary conditions for $\Phi\colon \mathscr{D}\to\mathscr{C}^T$ to be monadic.

It uses {\it reflexive pairs} in $\mathscr{D}$ instead of $U$-split pairs. 

A pair of morphisms $f,g\colon A\to B$ in $\mathscr{D}$ is called {\it reflexive} if there is a morphism $h\colon B\to A$ which splits both $f$ and $g$: $f\circ h=\id_B=g\circ h$. 

We refer the reader to [MLM, Ch.IV.4, Th.2] for a proof of the following result, also known as the {\it crude monadicity Theorem}:
\begin{theorem}\label{montheorem2}
Let $F\colon \mathscr{C}\rightleftarrows\mathscr{D}\colon U$ be a pair of adjoint functors, and let $T=UF$ be the corresponding monad. Consider the Eilenberg-MacLane comparison functor $\Phi\colon \mathscr{D}\to\mathscr{C}^T$. Then:
\begin{itemize}
\item[(1)] if $\mathscr{D}$ has coequalizers of all reflexive pairs, the functor $\Phi$ admits a left adjoint $\Psi\colon\mathscr{C}^T\to\mathscr{D}$,
\item[(2)] if, furthermore, $U$ preserves these coequalizers, the unit of the adjunction $\Id_{\mathscr{C}^T}\to \Phi\circ \Psi$ is an isomorphism,
\item[(3)]  if, furthermore, $U$ reflects isomorphisms, the counit of the adjunction $\Psi\circ\Phi\to\Id_{D}$ is also an isomorphism.
\end{itemize}
Therefore, if (1)-(3) hold, $(U,F)$ is monadic.
\end{theorem}
Note that, unlike for Theorem \ref{montheorem1}, {\it the converse statement is not true}. That is, the conditions for monadicity, given in Theorem \ref{montheorem2}, are sufficient but not necessary.

The following construction is of fundamental importance for both monadicity theorems. 

In the notations as above, let $A\in\mathscr{D}$. Consider two morphisms 
\begin{equation}\label{eqmonmain}
FUFUA\overset{f}{\underset{g}{\rightrightarrows}}FUA
\end{equation}
where $f=FU\varepsilon_A$ and $g=\varepsilon_{FUA}$. (Similarly, one defines such two maps for $A\in \mathscr{C}^T$).

One has {\it two different extensions} of this pair of arrows, which form a $U$-split coequalizer and a reflexive pair, correspondingly. 

For the first case, consider 
\begin{equation}\label{moncase1}
\xymatrix{
UFUFUA\ar@<0.5ex>[r]^{Uf}\ar@<-0.5ex>[r]_{Ug}&UFUA\ar[r]^{h}\ar@/_1.5pc/[l]_{s_1}&UA\ar@/_1.5pc/[l]_t
}
\end{equation}
with $s_1=\eta_{UFUA}$, $t=\eta_{UA}$, $h=U\varepsilon_A$.

For the second case, consider 
\begin{equation}\label{moncase2}
\xymatrix{
FUFUA\ar@<0.5ex>[r]^{f}\ar@<-0.5ex>[r]_{g}&FUA\ar@/_1.5pc/[l]_{s_2}
}
\end{equation}
with $s_2=F\eta_{UA}$.

The following lemma is proven by a direct check:
\begin{lemma}\label{lemmasplitrefl}
For any $A\in\mathscr{D}$ (or $A\in\mathscr{C}^T$), \eqref{moncase1} is a split coequalizer in $\mathscr{C}$, whence \eqref{moncase2} is a reflexive pair in $\mathscr{D}$ (corresp., in $\mathscr{C}^T$).
\end{lemma}
\qed

Note that $s_1$ is {\it not} a $U$-image of a morphism in $\mathscr{D}$, though $Uf$ and $Ug$ are. 
On the other hand, $s_2$ is a morphism in $\mathscr{D}$ (corresp., in $\mathscr{C}^T$).

\subsubsection{\sc The dg operad $\mathcal{O}$ and the monad of weakly unital dg categories}\label{subsectionmonadwu}
A dg graph $\Gamma$ over $\k$ is given by a {\it set} $V_\Gamma$ of vertices, and a complex $\Gamma(x,y)$ for any ordered pair $x,y\in V_\Gamma$. A morphism $F\colon \Gamma_1\to\Gamma_2$ is given by a map of sets $F_V\colon V_{\Gamma_1}\to V_{\Gamma_2}$, and by a map of complexes $F_E\colon \Gamma_1(x,y)\to \Gamma_2(F_V(x),F_V(y))$, for any $x,y\in V_{\Gamma_1}$. We denote by $\Graphs_\dgu(\k)$ the category whose objects are {\it unital} dg graphs over $\k$.

There is a natural forgetful functor $U\colon \Cat_{dgwu}(\k)\to \Graphs_\dgu(\k)$, where $U(C)$ is a graph $\Gamma$ with $V_\Gamma=\Ob(C)$, and $\Gamma(x,y)=C(x,y)$.

\begin{prop}\label{propmoncatwu}
The functor $U$ admits a left adjoint $F\colon \Graphs_\dgu(\k)\to \Cat_\dgwu(\k)$.
\end{prop}
\begin{proof}
We provide a construction of the right adjoint to $U$. 

Consider the non-$\Sigma$ the dg operad $\mathcal{O}$ define as the quotient-operad of the free operad generated by the composition operations:
\begin{itemize}
\item[(a)] the composition operation $m\in\mathcal{O}(2)^0$
\item[(b)] $p_{n; i_1,\
\dots,i_k}\in \mathcal{O}(n-k)^{-n+1}$, $0\le k\le n$, $1\le i_1<i_2<\dots<i_k\le n$, whose meaning is explained in \eqref{relop} below,
\item[(c)] a 0-ary operation $j\in \mathcal{O}(0)^0$ (which generates the morphisms $\id_x$, $x\in \Ob C$, for a weakly unital dg category $C$)
\end{itemize}
by the following relations:
\begin{equation}\label{orel}
\begin{aligned}
\ &(i) \text{ the associativity of  $m$, and $dm=0$}\\
&(ii)\  m\circ (j,j)=j, dj=0\\
&(iii)\  p_{n; i_1,\dots,i_k}=0 \text{ if $k=0$ or $k=n$, $n\ge 2$}\\
&(iv)\  p_{1;1}=j, p_{1; -}=\id\\
&(v)\  \text{ relation $\eqref{relop2}$ below}
\end{aligned}
\end{equation}
Note that (ii) formally follows from the part of (iii), saying that $p_{n;1,2,\dots,n}=0$, and (v).

For a weakly unital dg category $C$, the operation $p_{n; i_1,\dots,i_k}(f_1,\dots,f_{n-k})$ is defined as 
\begin{equation}\label{relop}
p_n\big(f_1,\dots,f_{i_1-1}, \underset{i_1}{1_{x_1}},f_{i_1},\dots,f_{i_2-2},\underset{i_2}{1_{x_2}},f_{i_2-1},\dots,f_{i_3-3},\underset{i_3}{1_{x_3}},\dots\dots,\underset{i_k}{1_{x_k}},f_{i_k-k+1},\dots,f_{n-k}\big)
\end{equation}
where by $1_{x_i}$s are denoted the morphisms $1_{x_i}\in \k_C$ for the corresponding objects $x_i\in C$.

The operad $\mathcal{O}$ is freely generated by these operations and by $m$, with the relations being the associativity of $m$ and the relations on $p_{n;i_1,\dots,i_k}$ which express that \eqref{relop} are the summands of the Taylor components for the $A_\infty$ functor $p\colon C\oplus\k_C\to C$. These relations read:

\begin{equation}\label{relop2}
\begin{aligned}
\ &dp_{n;i_1,\dots,i_k}=\sum_{1\le \ell\le n-1}\pm m\circ (p_{\ell; i_1,\dots,i_{s(\ell)}},p_{n-\ell;i_{s(\ell)+1},\dots,i_k})+\\
&\sum_{r=1}^{n-1}\pm p_{n-1; j_1,\dots,j_{q(r)}}\circ (\id,\dots,\id,\underset{r}{m(a(r),a(r+1))},\id,\dots,\id)
\end{aligned}
\end{equation}
with the notations explained below.

We have to explain notations in \eqref{relop2}. By $s(\ell)$ is denoted the maximal $s$ such that $i_s\le \ell$;  $a(r)$ is equal to $\id$ if $r\not\in\{i_1,\dots,i_k\}$ and is equal to $j$ otherwise. Finally, $q(r)\in \{k,k-1,k-2\}$; $q(r)=k$ if neither $r,r+1$ are in $\{i_1,\dots,i_k\}$, and in this case $j_s=i_s$ for $i_s\le r$ and $j_s=i_s-1$ for $i_s>r$; $q(r)=k-1$ if either $r$ or $r+1$ are in $\{i_1,\dots,i_k\}$ but not both, in this case $j_s=i_s$ for $i_s<r$, and $j_s=i_{s+1}-1$ for $i_{s+1}>r$; finally, if both $r,r+1$ are in $\{i_1,\dots,i_k\}$ one sets $q(r)=k-2$ and $j_s=i_s$ for $i_s<r$, and $j_s=i_{s+2}-1$ for $i_{s+2}>r+1$.

The category $\Graphs_\dgu(\k)$ has a natural internal $\Hom$ in $\Vect_\dg(\k)$. 
We associate with a graph $\Gamma\in \Graphs_\dgu(\k)$ a 1-globular set enriched over $\Vect_\dg(\k)$, in the sense of Batanin [Ba], in a standard way. Namely, we set $X_0=V_\Gamma$, and $X_1=\prod_{x,y\in V_\Gamma}\Gamma(x,y)_+$, where $\Gamma(x,y)_+=\Gamma(x,y)$ for $x\ne y$, and $\Gamma_+(x,x)=\Gamma(x,x)\oplus \k\id_x$. There are maps $t_0,t_1\colon X_1\to X_0$, mapping an element in $\Gamma(x,y)$ to $x$ and $y$, correspondingly, and a map $s\colon X_0\to X_1$ sending $x$ to $\id_x$. It is an (enriched) 1-globular set, meaning that $t_1s=t_2s=\id_{X_0}$. 
Therefore, one can talk on algebras in $\Graphs_\dgu(\k)$ over a dg operad. 

A structure of a weakly unital dg category $C$ on its underlying graph $U(C)$ in $\Graphs_\dg(\k)$ is the same that an action of the operad $\mathcal{O}$ on $U(C)$.

Let $\Gamma$ be a dg graph. Define $F(\Gamma)$ to be the free $\mathcal{O}$-algebra generated by $\Gamma$. 
Explicitly, $F(\Gamma)$ is defined as follows. 

We define a chain of length $n$ in $\Gamma$ as an ordered set $x_0,x_1,\dots, x_n$. Denote by $\Gamma_n$ the set of all chains of length $n$ in $\Gamma$. For $c\in\Gamma_n$, set 
$$
\Gamma(c):=\Gamma(x_0,x_1)_+\otimes\Gamma(x_1,x_2)_+\otimes\dots\otimes \Gamma(x_{n-1},x_n)_+
$$
and 
$$
\Gamma(n)(x,y):=\sum_{\substack{{c\in \Gamma_n}\\{x_0(c)=x,x_n(c)=y}}}\Gamma(c)
$$
(for $n=0$ we set $\Gamma(0)(x,x)=\k\id_x$ and $\Gamma(0)(x,y)=0$ for $x\ne y$).
Set
\begin{equation}
\Gamma_{\mathcal{O}}(x,y):=\sum_{n\ge 0}\mathcal{O}(n)\otimes\Gamma(n)(x,y)
\end{equation}
It gives rise to a graph $\Gamma_{\mathcal{O}}\in\Graphs_{\dgu}(\k)$ with $V_{\Gamma_\mathcal{O}}=V_\Gamma$.
Clearly $\Gamma_\mathcal{O}$ is an algebra over the operad $\mathcal{O}$, and therefore it defines a weakly unital dg category $F(\Gamma)$ such that $UF(\Gamma)=\Gamma_{\mathcal{O}}$.

One has:
\begin{equation}
\Hom_{\Cat_{dgwu}(\k)}(F\Gamma, D)=\Hom_{\Graphs_\dgu(\k)}(\Gamma, U(D))
\end{equation}
which is natural in $\Gamma$ and $D$, and gives rise to the required adjunction. 
\end{proof}

The dg operad $\mathcal{O}$ plays an important role in our paper. For the proof of Theorem \ref{theorq2} it will be important to know its cohomology.
Despite the answer is easy to state, the computation is rather technical. We provide it in Section \ref{sectionmop}.

\begin{theorem}\label{theoremmop}
The dg operad $\mathcal{O}$ is quasi-isomorphic to the operad $\mathscr{A}ssoc_+$ of strictly unital associative algebras, by the map sending $m$ to $m$, $j$ to 1, and all $p_{n; n_1,\dots,n_k}, k\ge 2$ to 0. 
\end{theorem}

\subsubsection{\sc The coequalizers in $\Graphs_{\dgu}(\k)$}\label{subsectioncoeqgra}
It is standard that coequalizers, and, therefore, all small colimits exist in $\Graphs_{\dgu}(k)$.

Recall the construction.

Let 
\begin{equation}\label{coeqgra}
\Gamma_1\overset{f}{\underset{g}{\rightrightarrows}}\Gamma_2
\end{equation}
be a pair of morphisms in $\Graphs_{\dgu}(\k)$.

Define its coequalizer $\Gamma_{f,g}$ as a small graph in $\Graphs_{\dgu}(\k)$ whose set of objects is the coequalizer of the corresponding maps of the sets of objects
$$
\Ob(\Gamma_1)\overset{f}{\underset{g}{\rightrightarrows}}\Ob(\Gamma_2)
$$
It is the quotient set of $\Ob(\Gamma_2)$ by the equivalence relation generated by the binary relation $f(x)\mathrm{R}g(x)$, $x\in\Ob(\Gamma_1)$.

Let $[x]$ and $[y]$ be two equivalence classes. Define a complex $\Gamma_{f,g}([x],[y])$ as the coequalizer in $\Vect_\dg(\k)$ of
\begin{equation}\label{coeqgra2}
\bigoplus_{\substack{{w,z\in \Ob(\Gamma_1)}\\{[f(w)]=[g(w)]=[x]}\\
{[f(z)]=[g(z)]=[y]}}}\Gamma_1(w,z)\overset{f_*}{\underset{g_*}{\rightrightarrows}}      \bigoplus_{\substack{{a,b\in \Ob(\Gamma_2)}\\ {[a]=[x], [b]=[y]}}}\Gamma_2(a,b)
\end{equation}
where $f_*$ maps $\phi\in\Gamma_1(w,z)$ to $f(\phi)$, and $g_*$ maps it to $g(\phi)$.
If at least one class of $[x],[y]$ is not in the image of $f$ (which is the same that the image of $g$), we define source complex in \eqref{coeqgra2} as 0. 

It is easy to check that the constructed dg graph $\Gamma_{f,g}$ is a coequalizer of \eqref{coeqgra}.

\subsubsection{\sc The coequalizers in $\Cat_{\dgwu}(\k)$, I}
Consider a pair of maps of weakly unital dg categories
\begin{equation}\label{eq21}
A\underset{G}{\overset{F}{\rightrightarrows}} B
\end{equation}
It is not straightforward to find (or to prove existence of) its coequalizer. 

However, one always can find the coequalizer of the maps of graphs 
\begin{equation}\label{eq22}
U(A)\underset{U(G)}{\overset{U(F)}{\rightrightarrows}} U(B)\xrightarrow{\ell}\Coeq(U(F),U(G))
\end{equation}
as in Section \ref{subsectioncoeqgra}.
For some special diagrams \eqref{eq21}, the functor $U$ creates coequalizers, see below.  Afterwards, we reduce the general coequalizers \eqref{eq21} to these special ones, in Section \ref{subsectioncoeqcatii}.

\begin{defn}{\rm
We say that the diagram \eqref{eq21} is {\it good} if $\Ob(A)=\Ob(B)$, and both $F$ and $G$ are identity maps on the sets objects.
}
\end{defn}

Assume that \eqref{eq21} is good.
Then the graph $\Coeq(U(F),U(G))$, which is a particular case of general coequalizers \eqref{coeqgra} in $\Graphs_\dgu(\k)$, is especially simple.  It has the set of vertices equal to  $\Ob(A)=\Ob(B)$, and its morphisms are the quotient-complexes $$\Coeq(U(F),U(G))(X,Y)=B(X,Y)/(F(f)-G(f))_{f\in A(X,Y)}$$

\begin{lemma}\label{lemmaquotient}
Suppose we are given a diagram \eqref{eq21} which is good. Then
a weakly unital dg category structure $Q$ and a map of weakly unital dg categories $L\colon B\to Q$ such that
$$
A\underset{G}{\overset{F}{\rightrightarrows}} B\xrightarrow{L} Q
$$ 
is a coequalizer, and $U(Q)=\Coeq(U(F),U(G))$, $U(L)=\ell$, exist if and only if the following two conditions hold:
\begin{itemize}
\item[(1)] the subcomplexes $(F(f)-G(f))_{f\in A(X,Y)}$, $X,Y\in \Ob(A)$, form a two-sided ideal in $B$:
\begin{equation}
\ell(g\circ (F(f)-G(f))\circ g^\prime)=0
\end{equation}
for any morphism $f$ in $A$ and any morphisms $g,g^\prime$ in $B$ (such that the compositions are defined),
\item[(2)] 
\begin{equation}
\ell(p_n^B(g_1\otimes\dots g_k\otimes (g\circ (F(f)-G(f))\circ g^\prime)\otimes g_{k+1}\otimes\dots \otimes g_{n-1}))=0
\end{equation}
for $n\ge 2$, and any morphism $f$ in $A$ (some of $g_i$ are elements of $\k_B$).
\end{itemize}
In particular, the weakly unital dg category $Q$, if it exists, is uniquely defined (which means that in this case $U$ strictly creates the coequalizer).
\end{lemma}
It is clear.

\qed

Recall that diagram \eqref{eq21} is called reflexive if there exists $H\colon B\to A$ such that $FH=GH=\id_B$.

\begin{prop}\label{proprefl}
Assume we are given a good and reflexive diagram \eqref{eq21}. Then conditions (1) and (2) of Lemma \ref{lemmaquotient} are fulfilled. Consequently, the functor $U$ strictly creates the coequalizer. 
\end{prop}
\begin{proof}
Prove that (1) holds. One has:
\begin{equation}
\begin{aligned}
\ &\ell(g\circ (F(f)-G(f))\circ g^\prime)=\ell(g\circ F(f)\circ g^\prime)-\ell(g\circ G(f)\circ g^\prime)=\\
&\ell(FH(g)\circ F(f)\circ FH(g^\prime))-\ell(GH(g)\circ G(f)\circ GH(g^\prime))=\\
&\ell(F(H(g)\circ f\circ H(g^\prime))-\ell(G(H(g)\circ f\circ H(g^\prime))=0
\end{aligned}
\end{equation}
Prove that (2) holds. One has:
\begin{equation}
\begin{aligned}
\ &\ell(p_n^B(g_1\otimes\dots \otimes (g\circ (F(f)-G(f))\circ g^\prime)\otimes \dots \otimes g_{n-1}))=\\
&\ell(p_n^B(g_1\otimes\dots\otimes (g\circ F(f)\circ g^\prime)\otimes\dots\otimes g_{n-1}))-
\ell(p_n^B(g_1\otimes\dots\otimes (g\circ G(f)\circ g^\prime)\otimes\dots\otimes g_{n-1}))=\\
&\ell(p_n^B(FH(g_1)\otimes\dots\otimes (FH(g)\circ F(f)\circ FH(g^\prime))\otimes\dots\otimes FH(g_{n-1}))-\\&\ell(p_n^B(GH(g_1)\otimes\dots\otimes (GH(g)\circ G(f)\circ GH(g^\prime))\otimes\dots\otimes GH(g_{n-1})))=\\
&\ell(p_n^B(FH(g_1)\otimes\dots\otimes (F(H(g)\circ f\circ H(g^\prime))\otimes\dots\otimes FH(g_{n-1}))-\\&\ell(p_n^B(GH(g_1)\otimes\dots\otimes (G(H(g)\circ f\circ H(g^\prime)))\otimes\dots\otimes GH(g_{n-1})))\overset{*}{=}\\
&\ell(Fp_n^A(H(g_1)\otimes\dots\otimes (H(g)\circ f\circ H(g^\prime))\otimes\dots\otimes H(g_{n-1})))-\\
&\ell(Gp_n^A(H(g_1)\otimes\dots\otimes (H(g)\circ f\circ H(g^\prime))\otimes\dots\otimes H(g_{n-1})))=0
\end{aligned}
\end{equation}
where the equality marked by * follows from the fact that $F,G$ are functors of weakly unital dg categories and \eqref{eq10}.
\end{proof}

\subsubsection{\sc The coequalizers in $\Cat_{dgwu}(\k)$, II}\label{subsectioncoeqcatii}
In this Section, we closely follow the arguments in [Wo, Prop. 2.11]. We reproduce them here for completeness. 

We make use of the following lemma, due to [Li, pp. 77-78], and known as the 3x3-lemma. 

\begin{lemma}\label{lemma3x3}
Consider the following diagram in a category
\begin{equation}
\xymatrix{
A_1\ar@<0.5ex>[r]^{h_1}\ar@<-0.5ex>[r]_{h_2}\ar@<-0.5ex>[d]_{\alpha_1}\ar@<0.5ex>[d]^{\alpha_2}&B_1\ar[r]^{h_3}\ar@<-0.5ex>[d]_{\beta_1}\ar@<0.5ex>[d]^{\beta_2}&C_1\ar@<-0.5ex>[d]_{\gamma_1}\ar@<0.5ex>[d]^{\gamma_2}\\
A_2\ar@<0.5ex>[r]^{g_1}\ar@<-0.5ex>[r]_{g_2}\ar[d]_{\alpha_3}&B_2\ar@{}[dr]|{*}\ar[r]^{g_3}\ar[d]_{\beta_3}&C_2\ar[d]_{\gamma_3}\\
A_3\ar@<0.5ex>[r]^{f_1}\ar@<-0.5ex>[r]_{f_2}&B_3\ar[r]^{f_3}&C_3
}
\end{equation}
in which the top and the middle rows are coequalizers, the leftmost and the middle columns are coequalizers, and all squares commute: $g_i\alpha_i=\beta_i h_i$, $f_i\alpha_3=\beta_3 g_i$, $g_3\beta_i=\gamma_i h_3$, $f_3\beta_3=\gamma_3g_3$, $i=1,2$. Then the following statements are equivalent:
\begin{itemize}
\item[(1)] the bottom row is a coequalizer,
\item[(2)] the rightmost column is a coequalizer,
\item[(3)] the square in the lower right corner (marked by $*$) is a push-out.
\end{itemize}
\end{lemma}
\qed

\begin{prop}\label{propcoeqmain}
The category $\Cat_{dgwu}(\k)$ has all coequalizers. 
\end{prop}
\begin{proof}
Let 
\begin{equation}\label{abcoeq}
\xymatrix{
A\ar@<0.5ex>[r]^{H_1}\ar@<-0.5ex>[r]_{H_2}&B
}
\end{equation}
be two arrows in $\Cat_{dgwu}(\k)$ coequalizer of which we'd like to compute. Embed it to the following solid arrow diagram, where $(F,U)$ is the adjoint pair of functors from Proposition \ref{propmoncatwu}.
\begin{equation}\label{bigdiagram}
\xymatrix{
FUFUA\ar@<0.5ex>[rr]^{FUFU(H_1)}\ar@<-0.5ex>[rr]_{FUFU(H_2)}\ar@<-0.5ex>[dd]_{\epsilon_{FUA}}\ar@<0.5ex>[dd]^{FU\epsilon_A}&&
FUFUB\ar[rr]^{F(L^\prime)}\ar@<-0.5ex>[dd]_{\epsilon_{FUB}}\ar@<0.5ex>[dd]^{FU\epsilon_B}&&FE^\prime\ar@<-0.5ex>@{.>}[dd]_{\alpha_1}\ar@<0.5ex>@{.>}[dd]^{\alpha_2}\\
\\
FUA\ar@<0.5ex>[rr]^{FU(H_1)}\ar@<-0.5ex>[rr]_{FU(H_2)}\ar[dd]_{\epsilon_A}&&FUB\ar[rr]^{F(L)}\ar[dd]_{\epsilon_B}&&FE\ar@{.>}[dd]_{p}\\
\\
A\ar@<0.5ex>[rr]^{H_1}\ar@<-0.5ex>[rr]_{H_2}&&B\ar@{.>}[rr]^{q}&&X
}
\end{equation}
The upper and the middle rows are obtained from \eqref{abcoeq} by application of $FUFU$ and $FU$, correspondingly. Denote by $E$ the coequalizer of $(UH_1,UH_2)$ in $\Graphs_{\dgu}(\k)$, and by $E^\prime$ the coequalizer of $(UFUH_1,UFUH_2)$ in $\Graphs_{\dgu}(\k)$. As $F$ is left adjoint, $FE$ and $FE^\prime$ are the coequalizers of $(FUH_1,FUH_2)$ and $(FUFUH_1,FUFUH_2)$ in $\Cat_{\dgwu}(\k)$, correspondingly. Therefore, the upper and the middle rows of \eqref{bigdiagram} are coequalizers.

The leftmost and the middle columns fulfil the assumptions of Proposition \ref{proprefl}. Indeed, the upper pairs of arrows are reflexive, by the second case of Lemma \ref{lemmasplitrefl}, see \eqref{moncase2}. Therefore, these columns are coequalizers, by Proposition \ref{proprefl}. 

The dotted arrows $\alpha_1,\alpha_2$ are constructed as follows. For $\alpha_1$, consider the map 
$$
F(L)\circ \epsilon_{FUB}\colon   FUFUB\to FE
$$
The two compositions 
$$
FUFUA\overset{FUFUH_1}{\underset{FUFUH_2}{\rightrightarrows}} FUFUB\xrightarrow{F(L)\circ \epsilon_{FUB}}FE
$$
are equal, which gives rise to a unique map $\alpha_1\colon FE^\prime\to FE$.

Similarly, taking $FU\epsilon_B$ instead of $\epsilon_{FUB}$, one gets a unique map $\alpha_2\colon FE^\prime\to FE$, which coequalizes the corresponding two arrows. 

We claim that the pair $(\alpha_1,\alpha_2)$ is reflexive. We construct $\varkappa_E\colon FE\to FE^\prime$ such that
$\alpha_1\circ\varkappa_E=\alpha_2\circ \varkappa_E=\id_{FE}$.

Recall $\varkappa_A\colon FUA\to FUFUA$ and $\varkappa_B\colon FUB\to FUFUB$ given as in \eqref{moncase2}:
$$
\varkappa_A=F\eta_{UA},\ \ \varkappa_B=F\eta_{UB}
$$
These maps are sections of the corresponding pairs of maps, which make them reflexive pairs, see Lemma \ref{lemmasplitrefl}.
Consider
$$
F(L^\prime)\circ \varkappa_B\colon FUB\to FE^\prime
$$
The two maps 
$$
FUA\rightrightarrows FUB\xrightarrow{F(L^\prime)\circ \varkappa_B} FE^\prime
$$
are equal, which gives rise to a unique map
$$
\varkappa_E\colon FE\to FE^\prime
$$
A simple diagram chasing shows that $\alpha_1\circ \varkappa_E=\alpha_2\circ\varkappa_E=\id_{FE}$. 

One has $\Ob (FE)=\Ob(FE^\prime)$, and Proposition \ref{propcoeqmain} is applied. We get an arrow $p\colon FE\to X$ which is a coequalizer of $(\alpha_1,\alpha_2)$. 

Finally, we have to construct an arrow $q\colon B\to X$ making the square in the lower right corner commutative. To this end, consider 
$p\circ F(L)\colon FUB\to X$. The two compositions
$$
FUFUB\rightrightarrows FUB\xrightarrow{p\circ F(L)} X
$$
are equal, which gives a unique map $q\colon B\to X$. One checks that the lower right square commutes.

One makes use of Lemma \ref{lemma3x3} to conclude that the bottom row is a coequalizer. 

\end{proof}

We have already seen in Section \ref{limcolimfirst} that the products, the coproducts, and the equalizers  in $\Cat_{\dgwu}(\k)$ are constructed straightforwardly. Then Proposition \ref{propcoeqmain}, and the classic result [R, Th. 3.4.11] give:
\begin{theorem}\label{cococo}
The category $\Cat_{\dgwu}(\k)$ is small complete and small cocomplete. 
\end{theorem}
\qed

\subsubsection{\sc The monadicity}
Although  we will not be using the following result in this paper, it may have an independent interest. The argument is close to [Wo, Th. 2.13].
\begin{theorem}\label{wumonadic}
The adjunction
$$
F\colon \Graphs_{\dgu}(\k)\rightleftarrows \Cat_{dgwu}(\k)\colon U
$$
is monadic.
\end{theorem}
\begin{proof}
We deduce the statement from the Beck Monadicity Theorem \ref{montheorem1},
for which we have to prove that the assumptions in (1)-(3) in Theorem \ref{montheorem1} hold.

(1) has been proven in Proposition \ref{propcoeqmain}, by which $\Cat_{dgwu}(\k)$ has all coequalizers, and (3) is clear. 
One has to prove (2), that is, that the functor $U\colon \Cat_{dgwu}(\k)\to\Graphs_{\dgu}(\k)$ preserves all $U$-split coequalizers. We make use of Lemma \ref{lemma3x3}, once again.

Let a pair of arrows in $\Cat_\dgwu(\k)$ 
\begin{equation}
A\overset{H_1}{\underset{H_2}{\rightrightarrows}}B
\end{equation}
be $U$-split. 
Then
\begin{equation}\label{eqsplitagain}
UA\overset{UH_1}{\underset{UH_2}{\rightrightarrows}}UB\xrightarrow{L}E
\end{equation}
is a split coequalizer, for some $L$ and $E$. 
The upper and the middle rows in \eqref{bigdiagram} are defined now as the result of application of $FUF$ and $F$, correspondingly, to \eqref{eqsplitagain}. (In particular, now $E^\prime=UF(E)$, $L^\prime=UF(L)$). Therefore, the upper and the middle rows are split, and, therefore, {\it absolute} coequalizers, by Lemma \ref{lemmaac}.

Then we get the dotted arrows in \eqref{bigdiagram}, and construct $X$, as in the proof of Proposition \ref{propcoeqmain}. In particular, we get a coequalizer 
\begin{equation}
A\overset{H_1}{\underset{H_2}{\rightrightarrows}}B\xrightarrow{q}X
\end{equation}
at the bottom row of \eqref{bigdiagram}.
One has to prove that $UX\simeq E$. 

In the obtained diagram all columns and  two upper rows are split coequalizers, but the bottom row is also a coequalizer but possibly not split. Now apply the functor $U$ to the whole diagram. As split coequalizers are absolute, by Lemma \ref{lemmaac}, the upper two rows and all three columns remain coequalizers. Therefore, by Lemma \ref{lemma3x3}, the bottom row also remains a coequalizer, after application of the functor $U$.

\end{proof}

\comment

\section{\sc The cohomology of the dg operad $\mathcal{O}$}

\subsection{\sc }
Let $\Fin$ be the non-$\Sigma$ operad with $\Fin(n)=\k$ for any $n\ge 0$, with all compositions equal to the identity maps $\k\to\k$.
Let $\mathcal{O}$ be the non-$\Sigma$ operad defined in Section \ref{subsectionmonadwu}. There is a dg operad map $p\colon \mathcal{O}\to\Fin$, sending all operations $p_\ell(-)$ to 0.

We prove here
\begin{theorem}\label{theoracycl}
For any $n\ge 0$, the dg operad map $p$ defines a quasi-isomorphism $p(n)\colon \mathcal{O}(n)\to\Fin(n)$.
\end{theorem}

The proof is divided into several steps.

{\it In Step 1}, we consider a bigger operad $\hat{\mathcal{O}}$ endowed with a map of dg operads $P\colon \hat{\mathcal{O}}\to\mathcal{O}$, and prove that this map is a quasi-isomorphism.  Still we impose the relation $p_{n;i_1,\dots,i_k}=0$ for $k=0$. It is a dg operad generated by $m$, $j$, $p_{n;i_1,\dots,i_k}$, having the same arity and degree as for the operad $\mathcal{O}$ (see Section \ref{subsectionmonadwu}), with the following distinction: 
\begin{equation}
\text{ We do {\it not} impose the relations $p_{n;1,2,3,\dots,n}=0$ and $m\circ (j,j)=j$}
\end{equation}
The differentials $dj=0$, $dm=0$, and $dp_{n;i_1,\dots,i_k}$ is given by \eqref{relop2}.

The dg operad $\mathcal{O}$ is the quotient dg operad of the dg operad $\hat{\mathcal{O}}$ by the dg operadic ideal $I$ generated by the elements $\{p_{n;1,2,3,\dots,n}\}_{n\ge 2}$.

Note that 
\begin{equation}
dp_{2;1,2}=p_{1;1}-m(p_{1;1},p_{1;1})=j-m(j,j)
\end{equation}
Therefore, one has $j=m(j,j)$ in the quotient-operad. Denote by $P\colon \hat{\mathcal{O}}\to\mathcal{O}$ the quotient-map.

We prove
\begin{prop}\label{propop1}
The map of dg operads $P\colon \hat{\mathcal{O}}\to\mathcal{O}$ is a quasi-isomorphisms of dg operads.
\end{prop}
The proof is given in ??? below.

{\it In Step 2,} we consider the natural projection $P^\prime\colon \hat{\mathcal{O}}\to\Fin$ which sends all $p_{n;i_1,\dots,i_k}$, $n\ge 2$, to 0. We prove
\begin{prop}\label{propop2}
The map of dg operads $P^\prime\colon \hat{\mathcal{O}}\to\Fin$ is a quasi-isomorphism.
\end{prop}
The proof is given in ??? below. 

{\it In Step 3,} we prove Theorem \ref{theoracycl} by considering the roof of maps
\begin{equation}
\xymatrix{
&\hat{\mathcal{O}}\ar[dl]_{P}\ar[dr]^{P^\prime}\\
\mathcal{O}&&\Fin
}
\end{equation}
each of which is a quasi-isomorphism, by Propositions \ref{propop1} and \ref{propop2}.

\subsection{\sc A proof op Proposition \ref{propop1}}
Denote 
$$
p_n(1,\dots,1):=p_{n;1,2,\dots,n}
$$
the 0-ary operation of degree $-n+1$,
and
by $m_n$ the $n$-ary operation $m\circ_1(\dots \circ_1(m\circ_1(m\circ_1 m)))$ of degree 0. 

The dg operad ideal $I$ can be written as 
\begin{equation}\label{kun1}
I=\hat{\mathcal{O}}\circ I_0\circ \hat{\mathcal{O}}
\end{equation}
where $I_0\subset I(0)$ is
the coimplex with
\begin{equation}
I_0=\bigoplus_{N\ge 1}I_{0,N}
\end{equation}
and the complex $I_{0,N}$ is
\begin{equation}
\begin{aligned}
\ &0\to \underset{\deg=-N}{\k p_{N+1}(1,\dots,1)}
\to\underset{\deg=-N+1}{\oplus_{n_1+n_2=N+1}\k m_2(p_{n_1}(1,\dots,1),p_{n_2}(1,\dots,1))}\to\dots\to \\
&\underset{\deg=0}{\k \big(m_{N+1}(p_1(1),\dots,p_1(1))-m_{N}(p_1(1),\dots,p_1(1))\big)}\to 0
\end{aligned}
\end{equation}
Clearly it is enough to prove that $I$ is acyclic in all arities. By \eqref{kun1} and the  Kunneth formula, it is enough to prove
\begin{lemma}
For any $N\ge 2$, the complex $I_{0,N}$ is acyclic in all degrees. Consequently, $I_0$ is acyclic. 
\end{lemma}
\begin{proof}
\end{proof}

\subsection{\sc A proof of Proposition \ref{propop2}}

Theorem \ref{theoracycl} is proven.

\qed

\endcomment

\comment
\section{\sc Refined weakly unital dg categories}
Here we introduce the category $\Cat^\prime_\dgwu(\k)$ of {\it refined} weakly unital dg categories. We prove that the non-symmetric dg 1-operad $\mathcal{O}^\prime$, defining the corresponding monad,  is quasi-isomorphic to the operad $\Assoc_\u$ of unital associative algebras. The dg operad $\mathcal{O}$ itself (see Section \ref{subsectionmonadwu}) is likely quasi-isomorphic to $\Assoc_\u$ as well, but it is not proven in this paper.
On the other hand, the category $\Cat_\dgwu^\prime(\k)$ of refined weakly unital dg categories is more relevant for our main applications in Section ??? than the category $\Cat_\dgwu(\k)$. 
From this viewpoint, the category $\Cat^\prime_\dgwu(\k)$ is more interesting for us. 

The operad $\mathcal{O}^\prime$ was found when we scrutinized the explicit construction of a weakly unital dg category structure on $\Cobar(\Bar(C))$, see ???. Here we mean that the relations (N1)--(N3) below hold for $\Cobar(\Bar(C))$, and we postulated them for $\Cat_\dgwu^\prime(\k)$. The main advantage of it is that, for the closed model structure on $\Cat^\prime_\dgwu(\k)$, constructed in ???, $\Cobar(\Bar(C))$ is a {\it cofibrant object} for any unital dg category $C$. This important fact fails if $\Cobar(\Bar(C))$ is considered as an object of $\Cat^\prime_\dgwu(\k)$. 

\subsection{\sc The operad $\mathcal{O}^\prime$ and refined weakly unital dg categories}
Recall the operad $\mathcal{O}$ with generators $j$ (a 0-ary operation of degree 0), $m$ (a binary operation of degree 0), and $p_{n; i_1,\dots,i_k}$ (an $(n-k)$-ary operation of degree $-n+1$) with the relation and the differential listed in Section \ref{subsectionmonadwu}.

Impose two more sets of relations on the generators:
\begin{itemize}
\item[(N1)] 
\begin{equation}
p_{n_1;i_1,\dots,i_{k_1}}\circ_s p_{n_2;j_1,\dots,j_{k_2}}=
(-1)^?\ p_{n_1+n_2-1;i_1,\dots,i_t,j_1+s-1,\dots,j_{k_2}+s-1,i_{t+1}+n_2-1,\dots,i_{k_1}+n_2-1}
\end{equation}
where $n_1,n_2\ge 2$, $1\le s\le n_1$, $s\not \in\{i_1,\dots,i_k\}$, and $t$ is defined as maximal $q$ for which $i_q<s$,
\item[(N2)]
\begin{equation}
p_{n; i_1,\dots,i_k}\circ_s m=(-1)^?\ m\circ (p_{s;i_1,\dots,i_t},\  p_{n-s+1;i_{t+1}-s+1,\dots,i_k-s+1})
\end{equation}
where $n\ge 2$, $1\le s\le n$, $s\not\in\{i_1,\dots,i_k\}$, and $t$ is the maximal $q$ for which $i_q<s$.
\item[(N3)]
\begin{equation}
p_{n;i_1,\dots,i_k}\circ_sj=(-1)^?\ p_{n;i_1,\dots,i_\ell,s,i_{\ell+1},\dots,i_k}
\end{equation}
where $n\ge 2$, $k\ge 1$, $s\not\in \{i_1,\dots,i_k\}$, and $\ell$ is the maximal $q$ such that $i_\ell<s$.
\end{itemize}

\begin{lemma}\label{lemmarefined1}
The relations (N1)--(N3) agree with the differential defined in Section \ref{subsectionmonadwu}, and the corresponding graded quotient operad is in turn a dg operad. 
\end{lemma}
\begin{proof}
One has to prove that
\begin{equation}\label{dn1}
d\Big(p_{n_1;i_1,\dots,i_{k_1}}\circ_s p_{n_2;j_1,\dots,j_{k_2}}\Big)=(-1)^?d\Big(\ p_{n_1+n_2-1;i_1,\dots,i_t,j_1+s-1,\dots,j_{k_2}+s-1,i_{t+1}+n_2-1,\dots,i_{k_1}+n_2-1}\Big)
\end{equation}
and
\begin{equation}\label{dn2}
d\Big(p_{n; i_1,\dots,i_k}\circ_s m\Big)=(-1)^?d\Big(\ m\circ (p_{s;i_1,\dots,i_t},\  p_{n-s+1;i_{t+1}-s+1,\dots,i_k-s+1})\Big)
\end{equation}
hold in the graded quotient operad of $\mathcal{O}$ by (N1)--(N3),
where $d$ is given by \eqref{relop2} and by $dm=0$, and $d$ is extended by the Leibniz rule with respect to the operadic compositions. 

Both equations \eqref{dn1} and \eqref{dn2} are proved by straightforward but quite lengthy computations. We omit these computations here.
\end{proof}

We denote by $\mathcal{O}^\prime$ the quotient dg operad of $\mathcal{O}$ by relations (N1)--(N3). 

Define a {\it refined weakly unital dg category} as follows. 
\begin{defn}
A refined weakly unital dg category $C$ is a weakly unital dg category  (in the sense of Definition \ref{defwu}), such that the Taylor components $\{p_n\}$ of the $A_\infty$ functor $p$ obeys additionally the relations (N1)-(N3). A morphism of refined weakly unital dg categories is defined as a morphism of the underlying weakly unital dg categories.
\end{defn}
The category of refined weakly unital dg categories is denoted by $\Cat_\dgwu^\prime(\k)$.

\subsection{\sc The category $\Cat_\dgwu^\prime(\k)$ is small complete and small cocomplete}

\subsection{\sc The cohomology of $\mathcal{O}^\prime$}
We prove 
\begin{theorem}
The dg operad $\mathcal{O}^\prime$ is quasi-isomorphic to $\Assoc_\u$. 
\end{theorem}
\begin{proof}
Any element  $\omega\in\mathcal{O}^\prime(N)$ can be written down as a linear combination of products 
$m_\ell\circ (X_1,\dots,X_\ell)$, where each $X_i$ is of either of the following types:
\begin{itemize}
\item[(a)] the 0-ary operation $j$,
\item[(b)] the 1-ary operation $\id$,
\item[(c)] $p_{n; i_1,\dots,i_k}\circ (\id,\id, j,\id,j,j,\dots)$ for $n\ge 2$ (the right-hand side factor is a string of $n-k$ elements each of which is either $j$ or $\id$).
\end{itemize}
Moreover, such presentation is unique if we require it is reduced, meaning that the only ambiguity comes from the identity
$m\circ (j,j)=j$.

We consider further reductions of such expressions, when $\omega$ is a cycle.

Define $\mathcal{L}(N)$ as the following subcomplex in $\mathcal{O}^\prime(N)$. An element in $\mathcal{L}(N)$ is a linear combination of the monomials $m_\ell\circ (Y_1,\dots,Y_\ell)$, where each $Y_i$ is of either of the following types:
\begin{itemize}
\item[(1)] the 0-ary operation $j$,
\item[(2)] the 1-ary operation $\id$,
\item[(3)] $p_{n;1,2,\dots, n-1}\circ (j)$
\item[(4)] $p_{n; 2,3,\dots,n}\circ (j)$
\end{itemize}
(for the last cases, $p_{n; 1,2,\dots,n-1}$ and $p_{n;\ 2,3,\dots,n}$ are 1-ary operations; the composition with $j$ makes them 0-ary operations).

Consider the subcomplex $\mathcal{L}_1(N)$ of $\mathcal{L}(N)$ formed by linear combinations of the monomials $m_\ell(Y_1,\dots,Y_\ell)$ as above, for which the following condition holds:
\begin{equation}\label{condextra}
\text{For any $Y_i$ of types (2)-(4) both $Y_{i-1}$ and $Y_{i+1}$ are of type (1)}
\end{equation}
(In particular, if $Y_i$ is of type (2)-(4), one has $2\le i\le \ell-1$).

One checks that both $\mathcal{L}(N)$ and $\mathcal{L}_1(N)$ are indeed complexes.

\begin{lemma}\label{klemmaproofx}
Any cycle $\omega\in\mathcal{O}^\prime(N)$ is cohomologous to a cycle in $\mathcal{L}_1(N)$, $N\ge 0$.
\end{lemma}
\begin{proof}
We prove that $\omega$ is cohomologous to both $\omega\circ_i m(\id,j)$ and $\omega\circ_i m(j,\id)$, $1\le i\le N$. To this end, consider
$\omega\circ_i p_{2; 2}$ (correspondingly, $\omega\circ_i p_{2; 1}$). 

By the Leibniz rule and the assumption $d\omega=0$, one has 
$$
d(\omega\circ_i p_{2;2}=\pm \omega\circ_i (m(\id,j)-\id)=\pm(\omega-\omega\circ_i m(\id,j))
$$
The computation for $\omega\circ_i m(j,\id)$ is similar, replacing $p_{2;2}$ by $p_{2; 1}$.

Then (N2) applied to either of $\omega\circ_i m(\id,j)$ or $\omega\circ_i m(j,\id)$ ``splits'' each monomial summand of $\omega$ into the product of two smaller monomials. 

We repeat this procedure succesively.

Similarly, one replaces each instance of $j$ by $m(j,j)$, followed by application of (N2). 

Applying these steps succesively, replace $\omega$ by a cohomologous cycle, which is a sum of monomials $m_\ell\circ (Y_1,\dots,Y_\ell)$, where each $Y_k$ is one of the types (1)-(4) listed above. At the same, \eqref{condextra} not necesaarily holds.

-------------

\end{proof}

We turn back to the proof of Theorem.
It follows from Lemma \ref{klemmaproofx} that the natural map $i_*\colon H^\udot(\mathcal{L}_1(N))\to H^\udot(\mathcal{O}^\prime(N))$ is a surjection. We prove below that 
\begin{equation}\label{eqproofkey}
H^s(\mathcal{L}_1(N))=\begin{cases}
\k&s=0\\
0&\text{otherwise}
\end{cases}
\end{equation}
It would imply that $i_*$ is a quasi-isomorphism, as well as the statement of Theorem.

It remains to prove \eqref{eqproofkey}. 

By the K\"{u}nneth formula, we have to compute the cohomology of the following complexes $K_+$ and $K_-$:
\begin{equation}
K_+:\hspace{5mm}\dots\to\k\cdot (p_{4; 1,2,3}\circ j)\to \k\cdot (p_{3; 1,2}\circ j)\to \k\cdot (p_{2;1}\circ j)\to \underset{\deg=0}{\k\cdot j}\to 0
\end{equation}
\begin{equation}
K_-:\hspace{5mm}\dots\to\k\cdot (p_{4; 2,3,4}\circ j)\to \k\cdot (p_{3; 2,3}\circ j)\to \k\cdot (p_{2;2}\circ j)\to \underset{\deg=0}{\k\cdot j}\to 0
\end{equation}
with the differential inherited by the differential in $\mathcal{O}^\prime$.

Both  complexes are isomorphic to the complex 
\begin{equation}
\dots\to\k\xrightarrow{\sim}\k\xrightarrow{0}\k\xrightarrow{\sim}\k\xrightarrow{0}\underset{\deg=0}{\k}\to 0
\end{equation}
Therefore, they are quasi-isomorphic to $\k[0]$.

It proves \eqref{eqproofkey} and Theorem.

\end{proof}
\endcomment

\section{\sc A closed model structure on $\Cat_\dgwu(\k)$}
Here we construct a cofibrantly-generated closed model structure on the category $\Cat_\dgwu(\k)$. The construction generalises the Tabuada construction [Tab] of a cofibrantly-generated closed model structure on $\Cat_\dg(\k)$. Some arguments are new, such as  Lemma \ref{klemma} and Lemma \ref{lemmafibj}.

We assume some familiarity with closed model categories, in particular with [Ho, Ch.2].

\subsection{\sc A closed model structure on $\Cat_\dgwu(\k)$}
Denote by $\Assoc_+$ the $\k$-linear operad of unital associative algebras, $\Assoc_+(n)=\k$ for any $n\ge 0$, with standard operadic compositions.

\subsubsection{\sc The statement of the result}\label{section311}
Define {\it weak equivalences} $W$ in $\Cat_{dgwu}(\k)$ as the weakly unital dg functors $F\colon C\to D$
such that the following two conditions hold:
\begin{itemize}
\item[(W1)] for any two objects $x,y\in C$, the map of complexes $C(x,y)\to D(Fx,Fy)$ is a quasi-isomorphism of complexes,
\item[(W2)] the functor $H^0(F)\colon H^0(C)\to H^0(D)$ is an equivalence of $\k$-linear categories.
\end{itemize}
Note that for a weakly unital dg category $C$, the category $H^0(C)$ is strictly unital, and the functor $H^0(F)$ is well-defined, see Lemmas \ref{lemmatriv1} and \ref{lemmatriv2}.

Define  {\it fibrations} in $\Cat_\dgwu(\k)$ as the weakly unital dg functors $F\colon C\to D$ such that the following two conditions hold:
\begin{itemize}
\item[(F1)] for any two objects $x,y\in C$, the map of complexes $C(x,y)\to D(Fx,Fy)$ is component-wise surjective,
\item[(F2)] for any $x\in C$ and a closed of degree 0 arrow $g\colon Fx\to z$ in $D$ (where $z$ is an object of $D$, a priori not necessarily in the image of $F$), such that $g$ becomes an isomorphism in $H^0(C)$, there is an object $y\in C$, and a closed degree 0 map $f\colon x\to y$ inducing an isomorphism in $H^0(D)$, and such that $F(f)=g$ (in particular, $F(y)=z$).
\end{itemize}
We denote the class of all fibrations by $\Fib$.

Define a class of weakly unital dg functors $\Surj$. A weakly unital dg functor $F\colon C\to D$ belongs to $\Surj$ if $F$ is surjective on objects, and if (F1) holds.\footnote{Note the our use of notation $\Surj$ does not coincide with the one in [Tab].}

The lemma below is standard:
\begin{lemma}\label{lemmaverys}
A weakly unital dg functor $F\colon C\to D$ belongs to $\Fib\cap W$ if and only if it belongs to $\Surj\cap(W1)$.
\end{lemma}
\begin{proof}
It is clear that $\Surj\cap(W1)$ implies $\Fib\cap W$.
Conversely, assume $F$ obeys $\Fib\cap W$. One has to prove that $F$ is surjective on objects. From $(W2)$ we know that $H^0(F)$ is essentially surjective, that is, for any object $z$ in $D$ there is a homotopy equivalence $g\colon Fx\to z$. By (F2), there is a homotopy equivalence $f\colon x\to y$ such that $F(f)=g$. In particular, $F(y)=z$.
\end{proof}

One of our main results is:
\begin{theorem}\label{theoremmain1}
The category $\Cat_\dgwu(\k)$ admits a cofibrantly generated closed model structure whose weak equivalences and fibrations are as above, and whose sets of generating cofibrations and generating acyclic cofibrations are as it is defined in Section \ref{sectionijwu} below.
\end{theorem}

\subsubsection{\sc The sets $I$ and $J$}\label{sectionijwu}
Here we define sets $I$ and $J$ of morphisms in $\Cat_\dgwu(\k)$ which later are proven to be the sets of {\it generating cofibrations} and of {\it generating acyclic cofibrations} for the closed model structure, whose existence is stated in Theorem \ref{theoremmain1}.

\vspace{2mm}

{\sc The Kontsevich dg category}
\noindent

Denote by $\mathscr{K}$ the strictly unital dg category with two objects $0$ and $1$, whose morphisms are described by generators and relations, as follows:
\begin{itemize}
\item a closed degree 0 morphism $f\in\mathscr{K}(0,1)$ and a closed degree 0 morphism $g\in\mathscr{K}(1,0)$, 
\item degree -1 morphisms $h_0\in \mathscr{K}(0,0)$ and $h_1\in\mathscr{K}(1,1)$ such that
\begin{equation}\label{kcat1}
gf=\id_0+dh_0.\ \ fg=\id_1+dh_1
\end{equation}
\item degree -2 morphism $r\in\mathscr{K}(0,1)$ such that
\begin{equation}\label{kcat2}
dr=h_1f-fh_0
\end{equation}
\end{itemize}
This category was introduced by Kontsevich in [K, Lecture 6]. 

It was proven in [Dr, 3.7] that $\mathscr{K}$ is a (semi-free) resolution of the dg category which is the $\k$-linear envelope of the ordinary category with two objects 0 and 1, and two morphisms $f^\prime: 0\to 1$, $g^\prime: 1\to 0$ such that $gf=\id_0$, $fg=\id_1$. \footnote{It follows from this result that any cocycle of negative degree is a coboundary in the complexes of morphisms of $\mathscr{K}$. Clearly $h_0g-gh_1$ is a cycle of degree -1 in $\mathscr{K}(1,0)$. The reader may like to find, as an exercise, an explicit degree -2 morphism in $\mathscr{K}(1,0)$ whose boundary is $h_0g-gh_1$.}

On the other hand, Kontsevich proved in [K1, Lecture 6] the following fact:

Assume we are given a dg category $C$, and a closed degree 0 morphism $\xi\in C^0(x,y)$, which is a homotopy equivalence (that is, which descends to an isomorphism in $H^0(C)$). Then there is a (not unique) dg functor $F\colon \mathscr{K}\to C$ such that $F(f)=\xi$.

Lemma below shows that this property still holds, with minor changes, when $C$ is a weakly unital dg category:
\begin{lemma}\label{klemma}
Let $C$ be a weakly unital dg category, and $\xi\in C^0(x,y)$ be a closed degree 0 morphism, such that $[\xi]\in H^0(C)$ is a homotopy equivalence. Then there is a weakly unital dg functor $F\colon \mathscr{K}\to C$ such that $F(f)=\xi^\prime$, where $\xi^\prime\in C^0(x,y)$ is a closed degree 0 morphism such that $[\xi]=[\xi^\prime]$ in $H^0(C)$.
\end{lemma}
\begin{proof}
The proof uses basically the same computation as in Kontsevich's proof for strictly unital case, with some adjustments.

The problem is that $1_y\cdot \xi$ and $\xi\cdot 1_x$ may be distinct from $\xi$. Consider $\xi^\prime:=1_y\cdot \xi\cdot 1_x$. Then $1_y\cdot \xi^\prime=\xi^\prime\cdot 1_x=\xi^\prime$ (because $1_z 1_z=1_z$ for any $z$, see Definition \ref{defwu}). By assumption, there is degree 0 morphism $\eta\in C^0(y,x)$ which is inverse to $\xi$ (and, therefore, inverse to $\xi^\prime$ as well) in $H^0(C)$. We get:
\begin{equation}\label{kcat3}
\eta\cdot \xi^\prime=1_x+dh_x,\ \ \ \xi^\prime\cdot \eta=1_y+dh_y
\end{equation}
Set
$$
\eta^\prime=1_x\cdot\eta\cdot 1_y,\ h_x^\prime=1_x\cdot h_x\cdot 1_x,\ h_y^\prime=1_y\cdot h_y\cdot 1_y
$$
From \eqref{kcat3} we find
\begin{equation}\label{kcat4}
\eta^\prime \cdot \xi^\prime=1_x+dh_x^\prime,\ \ \ \xi^\prime\cdot \eta^\prime=1_y+dh_y^\prime
\end{equation}
However, \eqref{kcat2} (for the corresponding morphisms) may fail. 

The rest of the proof is as in [K1, Lecture 6]. Maintain $\xi^\prime,\eta^\prime,h_x^\prime$, and set
\begin{equation}\label{kcat5}
h_y^\pprime:=h_y^\prime-\xi^\prime\cdot  h_x^\prime\cdot  \eta^\prime -h_y^\prime\cdot \xi^\prime\cdot \eta^\prime
\end{equation}
It is checked directly that $(\xi^\prime,\eta^\prime,h_x^\prime,h_y^\pprime)$ satisfy \eqref{kcat1} and \eqref{kcat2}, with 
\begin{equation}\label{kcat6}
r=-h_y^\pprime\cdot \xi^\prime \cdot h_x^\prime+\xi^\prime\cdot  h_x^\prime \cdot h_x^\prime
\end{equation}
\end{proof}

\vspace{2mm}

{\sc The sets $I$ and $J$}

\noindent

Define, for any integral number $n$, the complex $D(n):=\Cone(\k[n]\xrightarrow{\id}\k[n])$. It is the complex 
$$
\k[n]\xrightarrow{\id}\k[n-1]
$$
Denote $S(n-1)=\k[n-1]$, and denote by $i\colon S(n-1)\to D(n)$ the natural imbedding of complexes. 

Denote by $\mathscr{A}$ the dg category with a single object $0$, and with $\mathscr{A}(0,0)=\k$. Denote by $\kappa$ the strictly unital dg functor $\kappa\colon \mathscr{A}\to\mathscr{K}$, sending $0$ to $0$. 

Denote by $\mathscr{B}$ the (strictly unital) dg category with two objects 0 and 1, such that $\mathscr{B}(0,0)=\k$, $\mathscr{B}(1,1)=\k$, $\mathscr{B}(0,1)=0$, $\mathscr{B}(1,0)=0$. 

Denote by ${P}(n)$ the dg category with two objects 0 and 1, and $P(n)(0,1)=D(n)$, $P(n)(0,0)=0$, $P(n)(1,1)=0$, $P(n)(1,0)=0$. Denote by $\mathscr{P}(n)$ the weakly unital dg category 
$$
\mathscr{P}(n)=FU(P(n))
$$
(the functors $F\colon \Graphs_\dg(\k)\to\Cat_\dgwu(\k)$, $U\colon \Cat_\dgwu(\k)\to\Graphs_\dg(\k)$ are defined in \ref{subsectionmonadwu}).

Denote by $\alpha(n)$ the (weakly unital) dg functor $\alpha(n)\colon \mathscr{B}\to\mathscr{P}(n)$ sending 0 to 0 and 1 to 1.

Denote by $C(n)$ the dg category with two objects 0 and 1, and with morphisms $C(n)(0,1)=S(n-1)$, $C(n)(0,0)=0$, $C(n)(1,1)=0$, $C(n)(1,0)=0$. Denote 
$$
\mathscr{C}(n)=FU(C(n))
$$
the corresponding weakly unital dg category.

Consider the morphism $b(n)\colon C(n)\to P(n)$ the map of dg categories, sending 0 to 0, 1 to 1, and such that $S(n-1)=C(n)(0,1)\to P(n)(0,1)=D(n)$ is the imbedding $i$. Define
$$
\beta(n)=FU(b(n))\colon \mathscr{C}(n)\to\mathscr{P}(n)
$$
It is a weakly unital dg functor. 

Let $Q\colon \varnothing \to \mathscr{A}$ be the natural dg functor.

\vspace{2mm}

\noindent

Let $I$ be a set of morphisms in $\Cat_\dgwu(\k)$ which comprises the dg functor $Q$ and the weakly unital dg functors $\beta(n)$, $n\in\mathbb{Z}$.

Let $J$ be a set of morphisms in $\Cat_\dgwu(\k)$ which comprises $\kappa$ and $\alpha(n)$, $n\in\mathbb{Z}$.

The sets $I$ and $J$ are referred to as the sets of {\it generating cofibrations} and of {\it generating acyclic cofibrations}, correspondingly.

\subsubsection{\sc The morphisms with the RLP with respect to $I$ and $J$}
The morphisms with $RLP$ with repsect to a set $S$ of morphisms is denoted by $S{-}\inj$.

A weakly unital dg functor $\mathscr{P}(n)\to D$, for $D$ in $\Cat_\dgwu(\k)$, is 1-to-1 corresponded to a morphism in $D$ of degree  $-n$. Similarly, a weakly unital dg functor $\mathscr{C}(n)\to D$ is 1-to-1 corresponded to a {\it closed} degree $-n+1$ morphism in $D$. It is straightforward.

Assume a weakly unital dg functor $f\colon C\to D$ has RLP with respect to all $\alpha(n)$, $n\in \mathbb{Z}$:
\begin{equation}
\xymatrix{
\mathscr{B}\ar[r]^{t_2}\ar[d]_{\alpha(n)}&C\ar[d]^{\phi}\\
\mathscr{P}(n)\ar[r]^{t_1}\ar@{.>}[ur]&D
}
\end{equation}
For the functor $\phi$ it means that any morphism in $D(\phi x,\phi y)$ is $\phi(q)$, for some $q\in C(x,y)$. 
That is, $\phi$ is surjective on morphisms.

Assume that a weakly unital dg functor $\phi\colon C\to D$ has RLP with respect to all $\beta(n)$, $n\in\mathbb{Z}$:
\begin{equation}
\xymatrix{
\mathscr{C}(n)\ar[r]^{t_2}\ar[d]_{\beta(n)}&C\ar[d]^{\phi}\\
\mathscr{P}(n)\ar[r]^{t_1}\ar@{.>}[ur]&D
}
\end{equation}
One deduces from this property that for any $x,y\in C$, the map of complexes $C(x,y)\to D(\phi x,\phi y)$ is component-wise surjective, and is a quasi-isomorphism. 

We summarize:
\begin{lemma}\label{lemmarlp1}
A weakly unital dg functor $\phi\colon C\to D$ has RLP with respect to all $\alpha(n)$, $n\in\mathbb{Z}$ if and only if $\phi$ obeys $(F1)$. A weakly unital dg functor $\phi\colon C\to D$ has RLP with respect to all $\beta(n)$, $n\in\mathbb{Z}$ if and only if $\phi$ obeys $(F1)\cap(W1)$.
\end{lemma}
In fact, we have proved the ``only if'' parts of both statements. The proofs of the ``if'' parts are standard and are left to the reader. 

\qed

\begin{prop}\label{proprlp2}
One has:
\begin{equation}
I{-}\inj=\Surj\cap (W1)=J{-}\inj\cap W
\end{equation}
\end{prop}
\begin{proof}
In virtue of Lemma \ref{lemmarlp1}, for the first identity it is enough to show that any $\phi$ having RLP with respect to $Q$ is surjective on objects, which is trivial.

For the second identity, we prove a statement which also will be used later.
\begin{lemma}\label{lemmafibj}
One has $\Fib=J{-}\inj$.
\end{lemma}

\noindent

{\it Proof of $J{-}\inj\subset \Fib$}:

\vspace{2mm}

(F1) follows from RLP with respect to $\alpha(n)$, $n\in\mathbb{Z}$, see Lemma \ref{lemmarlp1}.
Prove (F2). Let $\phi\colon C\to D$ be in $J{-}\inj$. Let $x$ be an object in $C$, and $\xi\colon f(x)\to z$ a homotopy equivalence. 
Consider $\xi^\prime=1_{z}\cdot\xi\cdot 1_{f(x)}$. By Lemma \ref{klemma}, there is a weakly unital dg functor $F\colon \mathscr{K}\to D$ such that $F(f)=\xi^\prime$. Then the RLP gives a weakly unital dg functor $\hat{F}\colon \mathscr{K}\to C$ such that $\phi\circ \hat{F}=F$. In particular, $\eta^\prime=\hat{F}(f)\in C(x,?)$ is a homotopy equivalence, such that $\phi(\eta^\prime)=\xi^\prime$.
Now $\xi^\prime-\xi=dt$, by (F1) there exists $t^\prime$ such that $\phi(t^\prime)=t$. Finally, set $\eta:=\eta^\prime-dt^\prime$. Then $[\eta]=[\eta^\prime]$, and $\phi(\eta)=\xi$. It completes the proof of (F2). 

\vspace{2mm}

{\it Proof of $\Fib\subset J{-}\inj$}:

\vspace{2mm}

Let $\phi\colon C\to D$ in $\Fib$. (F1) is equivalent to the RLP with respect to $\alpha(n)$, $n\in\mathbb{Z}$.
It remains to prove the RLP with respect to $\kappa$ for $\phi$. The proof is quite involved.

We are given a weakly unital dg functor $F\colon \mathscr{K}\to D$. Apply (F2) to $\xi=F(f)\in D(\phi(x), z)$, it gives $\eta^\prime\in C(x, y)$ of degree 0, which is homotopy equivalence, $\phi(y)=z$. Set $\eta=1_{y}\cdot \eta^\prime\cdot 1_x$. 
We will construct $\hat{F}\colon \mathscr{K}\to C$ such that $\phi\colon \hat{F}=F$ and $\hat{F}(f)=\eta$. 

\vspace{1mm}

To this end, we make use of a construction from [Dr, 3.7], which links the Kontsevich dg category $\mathscr{K}$ with the Drinfeld dg quotient (loc.cit.). Let $\mathscr{I}_0$ be the (strictly unital) dg category with two objects 0 and 1 and generated by a single morphism $f\in \mathscr{I}_0(0,1)$ of degree 0, $df=0$. Denote $\mathscr{I}:=\mathscr{I}_0^{\pretr}$ the pre-triangulated hull of $\mathscr{I}$ (see [Dr, 2.4]). Consider the object $\Cone(f)\in\mathscr{I}$, and define $\mathscr{J}$ as the full dg sub-category in $\mathscr{I}$ with a single object $\Cone(f)$. Consider the Drinfeld dg quotient $\mathscr{D}:=\mathscr{I}/\mathscr{J}$, and denote by $\mathscr{D}_0$ the full dg subcategory in $\mathscr{D}$ with objects 0 and 1. The following result is due to Drinfeld, loc.cit.:

\begin{lemma}\label{lemmadrk}
One has $\mathscr{D}_0=\mathscr{K}$.
\end{lemma}
We reconstruct the argument in Appendix A.

\qed

It gives rise to the following construction. Let $\mathscr{E}$ be a (strictly unital) dg category, $\xi\in \mathscr{E}(x,y)$ a closed degree 0 morphism which is a homotopy equivalence. One has a dg functor $F\colon \mathscr{I}_0\to \mathscr{E}$, $F(f)=\xi$. It gives rise to 
$F^\pretr\colon \mathscr{I}\to \mathscr{E}^\pretr$. Denote by $\mathscr{X}\subset\mathscr{E}^\pretr$ the full dg subcategory which has a single object $\Cone(\xi)$. One gets
$$
\mathscr{D}=\mathscr{I}/\mathscr{J}\to \mathscr{E}^\pretr/\mathscr{X}
$$
The fact the $\xi$ is a homotopy equivalence implies that one has a dg functor $\mathscr{E}^\pretr/\mathscr{X}\to \mathscr{E}^\pretr$, depending on a contraction of $\Cone(\xi)$. 

We get a dg functor 
$$
\mathscr{D}_0\to \mathscr{E}
$$
which is, in turn, a dg functor $\mathscr{K}\to \mathscr{E}$. Conversely, any dg functor $\mathscr{K}\to \mathscr{E}$ is obtained in this way.\footnote{It gives, in particular, a more conceptual replacement for the Kontsevich computation reproduced in Lemma \ref{klemma} (for its strictly unital case).}

If all our categories were strictly unital, we would make use of this construction, to prove that $\Fib\Rightarrow \kappa{-}\inj$, as follows. 

One has:
\begin{lemma}\label{lemmaminor}
Let $X$ be a dg category, $x\in X$ an object. Assume there are two degree -1 maps $h_1,h_2\in X^{-1}(x,x)$ such that $dh_i=\id_x$, $i=1,2$. Then there is $t\in X^{-2}(x,x)$ such that $dt=h_2-h_1$
\end{lemma}
It is true for $t=h_1h_2$.

\qed

A dg functor $F\colon \mathscr{K}\to D$, $F(f)=\xi$, amounts to the same that a contraction of $\Cone(\xi)$ in $D^\pretr$. That is, we get $h\in D^\pretr(\Cone(\xi),\Cone(\xi))$ such that $dh=\id_{\Cone(\xi)}$. We know from (F2) that $\Cone(\eta)$ is contractible. It gives rise to $\tilde{h}_1\in C^\pretr(\Cone(\eta),\Cone(\eta))$ such that $d\tilde{h}_1=\id_{\Cone(\eta)}$. We may have not $\phi(\tilde{h}_1)=h$.  In any case, $d(\phi(\tilde{h}_1))=\id_{\Cone(\xi)}$. By Lemma \ref{lemmaminor} one has $\phi(\tilde{h}_1)-h=dt$. By (F1), we lift $t$ to $\tilde{t}$, $\phi(\tilde{t})=t$. Set $\tilde{h}:=\tilde{h}_1-d\tilde{t}$. One has $d\tilde{h}=\id_{\Cone(\eta)}$ and $\phi(\tilde{h})=h$. It gives a lift of the dg functor $\hat{F}\colon \mathscr{K}\to C$ such that $\phi\circ \hat{F}=F$.

\vspace{3mm}

In the weakly unital case, this speculation should be adjusted. 

The main point is that, for a weakly unital dg category $C$ and for a morphism $\xi\colon x\to y$ in $C$, we can not define $\Cone(\xi)$.
Indeed, we want any object to have a weak unit. One checks that $1_{\Cone(\xi)}:=(1_x, 1_{y[-1]})$ satisfies $d1_{\Cone(\xi)}=0$
if and only if one has $f\cdot 1_x=1_y\cdot f$. It means that we can define $C^\pretr$ but it fails to be weakly unital, even if $C$ is.

For a weakly unital dg category $C$, denote by 
$C_u$ the dg subcategory of $C$, whose objects are $\Ob(C)$, and whose morphisms are those morphisms $f$ in $C$ for which $1\cdot f=f\cdot 1$. We consider $C_u$ as a unital dg category. 

If $\phi\colon C\to D$ is in $\Fib$, then $\phi_u\colon C_u\to D_u$ is also in $\Fib$, as follows from the argument above, with replacement of $f$ by $1\cdot f \cdot 1$. 

As $\mathscr{K}$ is strictly unital, a weakly unital dg functor $F\colon \mathscr{K}\to D$ defines a dg functor $F_u\colon \mathscr{K}\to  D_u$. Then we construct $\hat{F}\colon \mathscr{K}\to C_u$, as in the strictly unital case. It completes the proof.

\qed

Now the second identity is proved as follows.
One has $J{-}\inj\cap W=\Fib\cap W=\Surj\cap (W1)$ where the first identity follows from Lemma \ref{lemmafibj}, and the second one follows from Lemma \ref{lemmaverys}.

\end{proof}

\comment
For the second identity, assume $\phi\in J{-}\inj\cap W$, prove that $\phi\in \Surj\cap(W1)$. By Lemma \ref{lemmarlp1}, it is enough to prove that $\phi$ is surjective on objects. By $(W2)$, for any object $z\in D$ there is an object $x\in C$ and a homotopy equivalence $\xi\colon \phi(x)\to z$. By Lemma \ref{klemma}, there is $\xi^\prime\in D(\phi x,z)$ such that $[\xi^\prime]=[\xi]$, and a dg functor $F\colon \mathscr{K}\to D$ such that $F(f)=\xi^\prime$. Then the RLP with respect to $\kappa\colon \mathscr{A}\to\mathscr{K}$ shows that $\xi^\prime=\phi(\eta)$ for some morphism $\eta\in C(x,?)$. In particular, there is an object $y\in C$ (for which $\eta\in C(x,y)$) such that $\phi(y)=z$. Therefore, $\phi$ is surjective on objects.

Conversely, assume $\phi\colon C\to D$ is in $\Surj\cap(W1)$. It is clear that $\phi$ has $(W2)$, so $\phi\in W$. It remains to show that $\phi$ has RLP with respect to $\kappa$. 

To this end, let $F\colon \mathscr{K}\to D$ be a (weakly unital) dg functor, so  that $F(f)=\xi\in D(x_1,y_1)$ is a closed degree 0 morphism. Then $\xi$ is a homotopy equivalence. By (W1), there is a closed degree 0 morphism $\eta\in C(x,y)$ such that $\phi(\eta)=\xi+d\lambda_1$. By (F1), $\lambda_1=\phi(\lambda)$, so that 
$
\phi(\eta-d\lambda)=\xi
$
Set $\eta^\prime=1\cdot (\eta-d\lambda)\cdot 1$. 
We still have $\phi(\eta^\prime)=\xi$. 
\endcomment

\subsubsection{\sc The proof of Theorem \ref{theoremmain1}}
The proof relies on [Ho, Th. 2.1.19]. 
Recall this theorem in a slightly different form, adapted for our needs:
\begin{theorem}\label{theoremhovey}
Let $\mathscr{C}$ be a category with all small limits and colimits. Suppose $\mathscr{W}$ is a subcategory of $\mathscr{C}$, and $I$ and $J$ are sets of maps. Assume the following conditions hold:
\begin{itemize}
\item[1.] the subacategory $W$ has two out of three property and is closed under retracts,
\item[2.] the domains of $I$ are small relative to $I{-}\cell$,
\item[3.] the domains of $J$ are small relative to $J{-}\cell$,
\item[4.] $J{-}\cell\subset \mathscr{W}\cap I{-}\cof$,
\item[5.] $I{-}\inj=W\cap J{-}\inj$.
\end{itemize}
Then there is a cofibrantly generated closed model structure on $\mathscr{C}$, for which the morphisms $W$ of $\mathscr{W}$ are weak equivalences, $I$ are generating cofibrations, $J$ are acyclic generating cofibrations. Its fibrations are defined as $J{-}\inj$.
\end{theorem}
The reader is referred to [Ho, Sect.2.1] for notations $S{-}\cof$ and $S{-}\cell$.

\qed

Prove Theorem \ref{theoremmain1}.

We check conditions (1)-(5) of Theorem \ref{theoremhovey}.
(1)-(3) are clear. We proved (5) in Proposition \ref{proprlp2}. It follows from (5) that $I{-}\inj\subset J{-}\inj$, therefore, $I{-}\cof\supset J{-}\cof$. Therefore, it remains to prove the part $J{-}\cell\subset W$ of (4), which we do below. The fact that  $J{-}\inj$ coincides with the class $\Fib$ defined in Section \ref{section311} is proven in Lemma \ref{lemmafibj}.

\vspace{2mm}

{\it Proof of $J{-}\cell\subset W$}:
We have to prove that in the following push-outs squares in $\Cat_\dgwu(\k)$ the weakly unital dg functor $f\colon \mathscr{X}\to \mathscr{Y}$ is a weak equivalence:
\begin{equation}\label{jweq}
(a)\ \xymatrix{
\mathscr{B}\ar[r]^{g}\ar[d]_{\alpha(n)}&\mathscr{X}\ar[d]^{f}\\
\mathscr{P}(n)\ar[r]& \mathscr{Y}
}\hspace{15mm}
(b)\ \xymatrix{
\mathscr{A}\ar[r]^{h}\ar[d]_{\kappa}&\mathscr{X}\ar[d]^{f}\\
\mathscr{K}\ar[r]& \mathscr{Y}
}
\end{equation}
where the (weak unital) dg functors $g$ and $h$ are arbitrary. We consider the cases (a) and (b) separately.

{\it The case (a):} It is clear that $\Ob(\mathscr{X})=\Ob(\mathscr{Y})$, and $f$ acts by the identity map on the objects. Therefore, we have to show that, for any objects $a,b\in\mathscr{X}$, the map of complexes $f(a,b)\colon \mathscr{X}(a,b)\to\mathscr{Y}(a,b)$ is a quasi-isomorphism. For objects 0 and 1 in $\mathscr{B}$, denote $u=g(0), v=g(1)$. Then
\begin{equation}
\begin{aligned}
\ &\mathscr{Y}(a,b)=\\
&\mathscr{X}(a,b)\bigoplus \mathcal{O}(3)\otimes X(a,u)\otimes D(n)\otimes X(v,b)\bigoplus 
\mathcal{O}(5)\otimes X(a,u)\otimes D(n)\otimes \mathscr{X}(v,u)\otimes D(n)\otimes \mathscr{X}(v,b)\bigoplus\dots
\end{aligned}
\end{equation}
where $\mathcal{O}$ is the operad introduced in \ref{subsectionmonadwu}.
The map $f(a,b)$ sends $\mathscr{X}(a,b)$ to the first summand. All other summands have 0 cohomology by the K\"{u}nneth formula, because $D(n)$ is acyclic.

{\it The case (b):} In this case, $\Ob(\mathscr{Y})=\Ob(\mathscr{X})\sqcup 1_{\mathscr{K}}$. 
It is clear that $H^0(f)$ is essentially surjective. One has to prove that $f$ is locally quasi-isomorphism:
$\mathscr{X}(a,b)\xrightarrow{quis}\mathscr{Y}(a,b)$, $a,b\ne 1_{\mathscr{K}}$. 
Denote $h(0_{\mathscr{A}})=u$.

By [Dr, 3.7], one knows that $\mathscr{K}$ is a resolution of the $\k$-linear envelope of the ordinary category with two objects 0 and 1, and with only morphism between any two objects. In particular, $\mathscr{K}(0,0)$ is quasi-isomorphic to $\k[0]$. Therefore, one can decompose (as a complex):
\begin{equation}
\mathscr{K}(0,0)=\bar{K}\oplus \k[0]
\end{equation}
where $\bar{K}$ is a complex acyclic in all degrees. At the same time, $\k[0]$ is corresponded to a morphism in $h(\mathscr{A}(0,0))\in\mathscr{X}(u,u)$; thus it is not a ``new morphism''.

One has:
\begin{equation}\label{eqdirsum}
\begin{aligned}
\ &\mathscr{Y}(a,b)=\mathscr{X}(a,b)\bigoplus \\
&\mathcal{O}(3)\otimes \mathscr{X}(a,u)\otimes \bar{K} \otimes \mathscr{X}(u,b)\bigoplus\\
&\mathcal{O}(5)\otimes \mathscr{X}(a,u)\otimes \bar{K}\otimes \mathscr{X}(u,u)\otimes  \bar{K}\otimes \mathscr{X}(u,b)\bigoplus\dots
\end{aligned}
\end{equation}
Note that \eqref{eqdirsum} is a direct sum of {\it complexes}.

The map of complexes $f(a,b)$ maps $\mathscr{X}(a,b)$ to the first summand. All other summands have 0 cohomology, because $\bar{K}$ is acyclic by [Dr, 3.7], and by the K\"{u}nneth formula. 

Note that we did not use Theorem \ref{theoremmop} here, the proof does not rely on a computation of the cohomology of the dg operad $\mathcal{O}$.

Theorem \ref{theoremmain1} is proven.

\qed

\section{\sc A Quillen equivalence between $\Cat_\dg(\k)$ and $\Cat_\dgwu(\k)$}
\subsection{\sc A Quillen pair}
Let $\mathscr{C}_1,\mathscr{C}_2$ be closed model categories. Recall that a {\it Quillen pair} of functors $L:\mathscr{C}_1\rightleftarrows \mathscr{C}_2:R$ is an adjoint pair of functors with an extra condition saying that 
$L$ preserves cofibrations and trivial cofibrations, or, equivalently, $R$ preserves fibrations and trivial fibrations. Either of these conditions
guarantee that a Quillen pair of functors descends to a pair of adjoint functors 
\begin{equation}\label{qe}
L:\Ho(\mathscr{C}_1)\rightleftarrows\Ho(\mathscr{C}_2):R
\end{equation}
between the homotopy categories, see e.g. [Hi, Sect. 8.5] or [Ho, Sect. 1.3].

In the case when $\mathscr{C}_1$ is cofibrantly generated, there is a simpler criterium [Ho, Lemma 2.1.20] for a pair of adjoint functors to be a Quillen pair. We reproduce it here for reader's convenience.
\begin{prop}\label{propho2}
Let $\mathscr{C}_1,\mathscr{C_2}$ be closed model categories, with $\mathscr{C}_1$ cofibrantly generated with generating cofibrations $I$ and generating acyclic cofibrations $J$. Let $L:\mathscr{C}_1\rightleftarrows \mathscr{C}_2:R$ be an adjoint pair of functors. Assume that $L(f)$ is a cofibration for all $f\in I$, and $L(f)$ is a trivial cofibration for all $f\in J$. Then the pair $(L,R)$ is a Quillen pair. 
\end{prop}
See [Ho, Lemma 2.1.20] for a proof.

\qed

Let $C$ be a weakly unital dg category. Define
$$
L(C)=C/I
$$
where $I$ is the dg category-ideal generated by $p_n(x_1,\dots,x_n)$, $x_i\in C\oplus\k_C$, $n\ge 2$. (Recall that $p_n(x_1,\dots,x_n)=0$ if $n\ge 2$ and all $x_i$ belong to $C\subset C\oplus \k_C$). 
Clearly $L(C)$ is a unital dg category.

The assignment $C\rightsquigarrow L(C)$ gives rise to a functor $L\colon \Cat_{\dgwu}(\k)\to\Cat_{\dg}(\k)$.

Let $A$ be a unital dg category. Define 
$$
R(A)=(A\oplus\k_A,p_\dg)
$$
where $p_\dg\colon A\oplus\k_A\to A$ is the dg functor constructed in Example \ref{uwu}. Recall that $p_\dg(1_x)=\id_x$, $x\in A$.
It gives rise to a functor
$R\colon \Cat_{\dg}(\k)\to\Cat_{\dgwu}(\k)$.

\begin{prop}\label{propqpair}
The following statements are true:
\begin{itemize}
\item[(1)] there is an adjunction 
$$
\Hom_{\Cat_{\dg}(\k)}(L(C),A)\simeq\Hom_{\Cat_{\dgwu}(\k)}(C,R(A))
$$
\item[(2)] the functors
$$
L\colon \Cat_{\dgwu}(\k)\rightleftarrows \Cat_{\dg}(\k)\colon R
$$
form a Quillen pair of functors. 
\end{itemize}
\end{prop}
\begin{proof}
(1): any map $F\colon C\to R(A)$ in $\Cat_\dgwu(\k)$ sends $p_n^C(-,\dots,-)$ , $n\ge 2$ to 0, because $C$ is strictly unital, see \eqref{eq10}. Therefore, this map is the same that a map $L(C)\to A$ in $\Cat_\dg(\k)$.

(2): Clearly $\{L(\beta(n)), L(Q)\}$ form the set $I_T$ of generating cofibrations for the Tabuada closed model structure [Tab], and $\{L(\alpha(n)), L(\kappa)\}$ for the set $J_T$ of generating trivial cofibrations for this model structure. The statement follows from Proposition \ref{propho2}.
\end{proof}

\subsection{\sc }
Recall that a Quillen pair $L\colon \mathscr{C}_1\rightleftarrows \mathscr{C}_2\colon R$ is called a {\it Quillen equivalence} if the following condition holds:

For all cofibrant $X\in\mathscr{C}_1$ and all fibrant $Y\in\mathscr{C}_2$ a morphism $f\colon LX\to Y$ is a weak equivalence in $\mathscr{C}_2$ if and only if the corresponding morphism $g\colon X\to RY$ is a weak equivalence in $\mathscr{C}_1$, see e.g. [Hi, Sect. 8.5.19], [Ho, Sect. 1.3.3].

Recall that this condition implies that the corresponding adjoint pair between the homotopy categories \eqref{qe} is an adjoint {\it equivalence} of categories. 

\begin{theorem}\label{theorq2}
The Quillen pair of functors 
$$
L\colon \Cat_{\dgwu}(\k)\rightleftarrows \Cat_{\dg}(\k)\colon R
$$
is a Quillen equivalence. 
\end{theorem}
\begin{proof}
Let $X\in \Cat_\dgwu(\k)$ be cofibrant, and $Y\in\Cat_\dg(\k)$ fibrant (therefore, $Y$ is an arbitrary object). On has to prove that $f\colon LX\to Y$ is a weak equivalence iff the adjoint map $f^*\colon X\to RY$ also is. 

It is enough to prove the statement for the case when $X$ is an $I$-cell.
Indeed, by the small object argument, for any $X$ there exist an $I$-cell $X^\prime$ such that $p\colon X^\prime\to X$ is an acyclic fibration. The Quillen left adjoint $L$ maps the weak equivalences between cofibrant object to weak equivalences, by [Hi, Prop. 8.5.7]. Therefore, $L(p)\colon L(X^\prime)\to L(X)$ is a weak equivalence. There is a map $i\colon X\to X^\prime$ such that $p\circ i=\id$, given by the RLP. By 2-of-3 axiom, $i$ is a weak equivalence, and $L(i)$ also is. 

Assume $L(X)\xrightarrow{f} Y$ is a weak equivalence, then $L(X^\prime)\xrightarrow{L(p)\circ f} Y$ is also a weak equivalence. 
If we know that the adjoint map $(f\circ L(p))^*\colon X^\prime\to R(Y)$ is a weak equivalence, then the adjoint map $f^*=(f\circ L(p))^*\circ i$ is also a weak equivalence. The converse statement is proven similarly.

Consider the case when $X$ is an $I$-cell for $\Cat_{\dgwu}(\k)$. We reduce this case of the statement to Theorem \ref{theoremmop}.

Denote by $V$ the graded graph of generators of $X$. Prove that for any objects $x,x^\prime\in X$, $y\in Y$, the cone 
$L_1=\Cone(LX(x,x^\prime))\xrightarrow{f} Y(fx,fx^\prime))$ is acyclic iff the cone $L_2=\Cone(X(x,x^\prime)\xrightarrow{f^*}RY(f^*x,f^*x^\prime))$ is acyclic. Denote $\bar{O}=\Ker(P\colon \mathcal{O}\to\Assoc_+)$, where $P$ is the dg operad map sending all $p_{n;-}$ to 0. There is a canonical map $\omega\colon L_2\to L_1$, and $\Cone(\omega)$ is quasi-isomorphic $F_{\bar{\mathcal{O}}}(V)(x,x^\prime)$, where $F_{\bar{\mathcal{O}}}(V)$ is the free algebra over $\bar{\mathcal{O}}$ generated by $V$, with an extra differential coming from the differential in the $I$-cell $X$. By Theorem \ref{theoremmop}, $\bar{O}$ is acyclic. 
Therefore, $F_{\bar{\mathcal{O}}}(V)$ is acyclic by the K\"{u}nneth formula. Therefore, $\Cone(\omega)$ is acyclic, and $L_1$ is quasi-isomorphic to $L_2$. Therefore, $L_1$ is acyclic iff $L_2$ is.

\end{proof}

\section{\sc A proof of Theorem \ref{theoremmop}}\label{sectionmop}

\subsection{\sc The dg operad $\mathcal{O}^\prime$ and its cohomology}
Recall that the dg operad $\mathcal{O}$ is generated by an $n$-ary operations $p_{n;n_1,\dots,n_k}$, acting as \\$p_{n}(f_1,\dots, f_{n_k-1},\underset{n_k}{1},f_{n_k+1},\dots)$, a binary operation $m$, with the relations and the differential as in \eqref{orel}. 

Define a dg operad $\mathcal{O}^\prime$, for which the dg operad $\mathcal{O}$ is a quotient-operad, as follows. The definition of $\mathcal{O}^\prime$ is similar to $\mathcal{O}$, but for the case of $\mathcal{O}^\prime$ we drop the relation $p_n(1,1,\dots,1)=0$ for $n\ge 2$, which holds in $\mathcal{O}$. 
We set $j=p(1)$, and thus $dp_2(1,1)=m(j,j)-j\ne 0$, $dp_3(1,1,1)=m(p(1),p_2(1,1))-m(p_2(1,1),p(1))$, and so on. 
The other relations and identities from \eqref{orel} remain the same.

There is a natural map of dg operads $P\colon \mathcal{O}^\prime\to\Assoc_+$, sending all $p_{n;\dots}$, $n\ge 2$, to 0. 
\begin{theorem}\label{theoremcomp1}
The map of dg operads $P\colon\mathcal{O}^\prime\to\Assoc_+$ is a quasi-isomorphism. 
\end{theorem}
\begin{proof}
Let $\omega\in\mathcal{O}^\prime$. Then $\omega$ is a linear combination of labelled ``trees'', where each vertex (excluding the leaves) is labelled either by $p_{n; n_1,\dots,n_k}$ or by $m$. We say that  $p_{n;n_1,\dots,n_k}$ has {\it $n-k$ operadic arguments} (the remaining $k$ arguments are 1's). We use notation $\sharp(p_{n;n_1,\dots,n_k})=n-k$. Given a tree $T$ in which a vertex $v$ is labelled by $p_{n;n_1,\dots,n_k}$, we write $\sharp(v)=n-k$. We extend $\sharp(-)$ to all vertices of $T$, by setting $\sharp(v)=0$ if $v$ is labelled by $m$. Denote by $V_T$ the set of all vertices of $T$ excluding the leaves.

For a given tree $T$, denote 
$$
\sharp(T)=\sum_{v\in V_T}\sharp(v)
$$
We also denote by $\sharp_p(T)$ the total number of vertices with $p_{\dots}$, {\it excluding} $p_1(1),p_2(1,1),\dots$. 

Define a descending filtration $F_\ldot$ on $\mathcal{O}^\prime$, as follows. Its $(-\ell)$-th term $F_{-\ell}$ is formed by linear combinations of labelled trees $T$ for which
$$
\sharp(T)-\sharp_p(T)\le \ell
$$
Note that for any tree $T$ one has $\sharp(T)-\sharp_p(T)\ge 0$.

One has:
$$
\dots\supset F_{-3}\supset F_{-2}\supset F_{-1}\supset F_0\supset  0
$$

Note that $dF_{-\ell}\subset F_{-\ell}$, and any component of the differential on $\mathcal{O}^\prime$ either preserves $\sharp(T)-\sharp_p(T)$ or decreases it by 1. 

We get a similar filtration $F_\ldot$ on the component $\mathcal{O}^\prime(N)$ of the airity $N$ operations. 

We compute cohomology of $\mathcal{O}^\prime(N)$ using the spectral sequence associated with filtration $F_\ldot$ on $\mathcal{O}^\prime(N)$. The spectral sequence lives in the quadrant $\{x\le 0,y\le 0\}$, the differential $d_0$ is horizontal. One easily sees that the spectral sequence converges. In fact, we show the spectral sequence collapses at the term $E_1$.

\begin{lemma}\label{lemmacomp1}
Consider the filtration $F_\ldot$ on $\mathcal{O}^\prime(N)$. One has:
$$E_1^{-\ell,m}=\begin{cases} \Assoc_+(N)& \ell=0,m=0\\0&\text{otherwise}\end{cases}$$
In particular, the spectral sequence collapses at the term $E_1$.
\end{lemma}
\begin{proof}
We write $p_{n;n_1,\dots,n_k}$ as $p_n(f_1,f_2,\dots, 1,\dots,f_{n-k})$ where $f_1,\dots,f_{n-k}$ are operadic arguments, and 1s stand on the places $n_1,n_2,\dots,n_k$. In these notations, describe the differential in $E_0^{-\ell,\ldot}=F_{-\ell}/F_{-\ell+1}$.

It has components of the following three types, which we refer to as Type I, Type II and Type III components.

\vspace{2mm}

{\it Type I components:} a component of Type I acts on a group of consequtive 1s, surrounded by operadic arguments from both sides, such as $$p_n(\dots,f_s,\underset{\text{a group of $i$ consequtive 1s}}{\underbrace{1,1,\dots,1}},f_{s+1},\dots)$$ For such a group, the component of $d_0$ is a sum of expressions, each summand of which is corresponded to either a product $1\cdot 1$ of two consequtive 1s, or to extreme products $f_s\cdot 1$ or $1\cdot f_{s+1}$, taken with alternated signs. It is clear that totally the component $d_0^S$ corresponded to such a group $S$ is equal to
$$
d_0^S(p_n(\dots, f_s,\underset{\text{$i$ of 1s in the group $S$}}{\underbrace{1,\dots,1}},f_{s+1},\dots))=\begin{cases}\pm p_n(\dots,f_s,\underset{\text{$i-1$ of 1s}}{\underbrace{1,1,\dots,1}},f_{s+1},\dots)&\text{if  $i$ is even}\\
0&\text{if  $i$ is odd}\end{cases}
$$

{\it Type II components:} a component of Type II acts on the groups of leftmost (corresp., rightmost) 1s, such as $p_n(1,1,\dots,1,f_1,\dots)$ or $p_n(\dots,f_{n-k},1,1,\dots,1)$, surrounded by an operadic argument from one side. There should be $\ge 1$ of 1s in the group for a non-zero result, and by assumption $p_n(\dots)$ contains at least one operadic argument. 

The corresponding component $d_0^S$ of the differential is a sum of two subcomponents: $d_0^S=d_0^{S,1}+d_0^{S,2}$. 

The first subcomponent $d_0^{S,1}=d_0^{S,1,-}\pm d_0^{S,1,+}$, where
$$
\begin{aligned}
\ &d_0^{S,1,-}(p_n(\underset{\text{$i$ of 1s}}{\underbrace{1,\dots,1}},f_1,\dots))=\\
&p_n(1\cdot 1,1,\dots,1,f_1,\dots)-p_n(1,1\cdot 1,\dots,f_1,\dots)+\dots+(-1)^{i-1} p_n(1,\dots,1,1\cdot f_1,\dots)
\end{aligned}
$$
and similarly for $d_0^{S,1,+}$ for the group of rightmost 1s. 

One has 
$$
d_0^{S,1,-}(p_n(\underset{\text{$i$ of 1s}}{\underbrace{1,\dots,1}},f_1,\dots))=\begin{cases}p_n(\underset{\text{$i-1$ of 1s}}{\underbrace{1,\dots,1}},f_1,\dots)&\text{if $i$ is odd}\\
0&\text{if $i$ is even}
\end{cases}
$$
and similarly for $d_0^{S,1,+}$.

The second subcomponent $d_0^{S,2}=d_0^{S,2,-}\pm d_0^{S,2,+}$, where 
$$
\begin{aligned}
\ &d_0^{S,2,-}(p_n(\underset{\text{$i$ of 1s}}{\underbrace{1,\dots,1}},f_1,\dots))=\\
&p_1(1)\cdot p_{n-1}(\underset{i-1}{1,\dots,1},f_1,\dots)-p_2(1,1)\cdot p_{n-2}(\underset{i-2}{1,\dots,1},f_1,\dots)+\dots+(-1)^{i-1}p_i(1,1,\dots,1)\cdot p_{n-i}(f_1,\dots)
\end{aligned}
$$
and similarly for $d_0^{S,2,+}$ for the rightmost group of 1s.

One checks that all other components of the differential $d$ on $\mathcal{O}^\prime$ decrease $\sharp(T)-\sharp_p(T)$ by 1.

{\it Type III components:} Here we have $d_0$ acting on $p_n(\underset{\text{$n$ of 1s}}{1,1,\dots,1})$.

One has:
\begin{equation}\label{d0d}
\begin{aligned}
\ &d_0(p_n(1,1,\dots,1))=\\
&p_{n-1}(1\cdot 1,1,\dots,1)-p_{n-1}(1,1\cdot 1,1,\dots,1)+\dots+(-1)^{i-1}p_{n-1}(1,1,\dots,1\cdot 1)+\\
& \pm\sum_{1\le i\le n-1}(-1)^{i-1}p_i(1,1,\dots,1)\cdot p_{n-i}(1,1,\dots,1)+
\end{aligned}
\end{equation}

Denote the first summand by $d_{0}^{S,1}$ and the second summand by $d_0^{S,2}$ One sees that
$$
d_0^{S,1}(p_n(1,1,\dots,1))=\begin{cases}
p_{n-1}(1,1,\dots,1)&\text{if $n$ is even}\\
0&\text{if $n$ is odd}
\end{cases}
$$

\vspace{2mm}

The computation of cohomology of the complex $(E_0^{-\ell,\ldot},d_0)$ is reduced to the computation of the cohomology of a tensor product of complexes (the factors are labelled by combinatorial data of the labelled tree $T$), corresponded to different components $S$ as listed above:
\begin{equation}
E_0^{-\ell,\ldot}=\bigotimes_{S,T}K^\udot_S
\end{equation}

The complexes $K_S$ corresponded to Type I components are isomorphic to 
\begin{equation}\label{complexk}
K^{\udot}=\{\dots\xrightarrow{0}\underset{i=4}{\k}\xrightarrow{\id}\underset{i=3}{\k}\xrightarrow{0}      \underset{i=2}{\k}\xrightarrow{\id}\underset{\deg=-1}{\underset{i=1}{\k}}\to 0\}
\end{equation}
The complex $K^\udot$ is acyclic in all degrees. It implies that the complex $(E_0^{-\ell,\ldot},d_0)$ is quasi-isomorphic to its subcomplex which is formed by the trees in which any $p$ is of the type\\ $p_n(1,1,\dots,1,f_1,\dots,f_{n-k},1,\dots,1)$, where all $n-k$ operadic arguments stand in turn, without 1s between them. 

It remains to treat the Type II and Type III cases. 

The complexes whose cohomology we need to compute are of two types. They are formed either by linear combinations of 
$$
p_{n_1}(1,1,\dots,1)\cdot p_{n_2}(1,1,\dots,1)\dots p_{n_k}(1,1,\dots,1)\cdot p_n(1,1,\dots,1,f_1,\dots)
$$
or by all linear combinations of
$$
p_{n_1}(1,1,\dots,1)\cdot p_{n_2}(1,1,\dots,1) {\dots} p_{n_k}(1,1,\dots,1)
$$
Denote them by $K_1^\udot$ and $K_2^\udot$. 

Their cohomology are computed similarly, we consider the case of $K_2^\udot$, leaving the case of $K_1^\udot$ to the reader.

Denote $p_\ell=p_\ell(1,1,\dots,1)$ and by $P_\ell$ the 1-dimensional vector space $\k p_\ell(1,1,\dots,1)=\k p_\ell$, $\ell\ge 1$.

One has:
$$
K_2^{-n}=\bigoplus_{k\ge 1,\ n_1+\dots+n_k-k=n}P_{n_1}\otimes P_{n_2}\otimes\dots\otimes P_{n_k}
$$
We denote the differential $d_0$ on $K_2^\udot$, see \eqref{d0d}, by $d$. 
\begin{lemma}\label{lemmacomp2}
The complex $(K_2^\udot, d)$ is quasi-isomorphic to $P_1[0]$. 
\end{lemma}
\begin{proof}
Consider on $K_2^\udot$ the following descending filtration $\Phi_\ldot$, where
$$
\Phi_{-\ell}=\bigoplus_{n_1+n_2+\dots+n_k\le \ell} P_{n_1}\otimes P_{n_2}\otimes\dots\otimes P_{n_k}
$$
One has
$$
\dots \supset\Phi_{-3}\supset \Phi_{-2}\supset \Phi_{-1}\supset \Phi_0= 0
$$
$$
d\Phi_{-\ell}\subset \Phi_{-\ell}
$$
Denote by $d_{0,\Phi}$ the differential in $E_{0,\Phi}^{-\ell,\ldot}=\Phi_{-\ell}/\Phi_{-\ell+1}$. It is given by 
\begin{equation}\label{diffcobar0}
d_{0,\Phi}(p_{n_1}\otimes p_{n_2}\otimes \dots\otimes p_{n_k})=\sum_{i=1}^k(-1)^{n_1+\dots+n_{i-1}-i+1}p_{n_1}\otimes\dots\otimes d_{0,\Phi}(p_{n_i})\otimes\dots\otimes p_{n_k}
\end{equation}
where
\begin{equation}\label{diffcobar}
d_0(p_n)=\sum_{1\le i\le n-1}(-1)^{i-1}p_i\otimes p_{n-i}
\end{equation}
It is well-known that the complex $E_{0,\Phi}^{-\ell,\ldot}$ is acyclic when $\ell\ge 2$, and is quasi-isomorphic to $P_1[0]$ when $\ell=1$.

We can identify $P_n\simeq (\k[1])^{\otimes n}$, then $\oplus_{n\ge 1}\k[1]^{\otimes n}=P$ becomes the (non-unital) cofree coalgebra cogenerated by $\k[1]$. The complex \eqref{diffcobar0}, \eqref{diffcobar} is identified with the cobar-complex of the cofree coalgebra $P$. It is standard that its cohomology is equal to $\k[1][-1]\simeq \k$. 

Therefore, the spectral sequence collapses at the term $E_1$ by dimensional reasons. 

It completes the proof of Lemma \ref{lemmacomp2}.

\end{proof}

Similarly we prove that $K_1^\udot$ is acyclic in all degrees.

In this way we see that any cohomology class in $E_0^{-\ell,\ldot}$ can be represented by a linear combination of trees which do not contain $p_n$s with $n\ge 2$. 

It follows that any cohomology class can be represented by a linear combination of trees containing only $m$ and $p(1)$, and all such trees have cohomological degree 0. 

It completes the proof. 

\end{proof}
Theorem \ref{theoremcomp1} immediately follows from Lemma \ref{lemmacomp1}.

\end{proof}

\subsection{\sc The cohomology of the dg operad $\mathcal{O}$}\label{subsectionterug}
We are to prove Theorem \ref{theoremmop}. 
\begin{proof}
The dg operad $\mathcal{O}$ is the quotient-operad of $\mathcal{O}^\prime$ by the dg operadic ideal $I$ generated by $p_n(1,\dots,1)$, $n\ge 2$. It is enough to prove that $I$ is acyclic. It would be natural to deduce the acyclicity of $I$ from the acyclicity of the complex $K_3^\udot=K_2^\udot/(\k p_1(1))$, established above, by application of the K\"{u}nneth formula. However, the K\"{u}nneth formula is not applicable, because we do {\it not} have a decomposition such as $I=\mathcal{O}^\prime \circ K_3^\udot\circ \mathcal{O}^\prime$, compatible with the differential. 

Alternatively, we repeat the arguments in the proof of Theorem \ref{theoremcomp1}. The main point is that the filtration $F_\ldot$ on $\mathcal{O}^\prime$,
defined in the course of the proof of Theorem \ref{theoremcomp1}, descends to $\mathcal{O}^\prime/I$. Indeed, both numbers $\sharp(T)$ and $\sharp_p(T)$ are well-defined on the quotient $\mathcal{O}^\prime/I$. The statement of Lemma \ref{lemmacomp1} holds in this case, and its proof follows the same line. It becomes even simpler, because for Type II and Type III summands we make use that $p_n(1,\dots,1)=0$ for $n\ge 2$, which substantially simplifies the computation. 
\end{proof}

\appendix 

\section{\sc The Drinfeld dg quotient and the Kontsevich dg category $\mathscr{K}$}
Here we reconstruct the proof of Lemma \ref{lemmadrk} sketched in [Dr, 3.7].

In this Appendix, we denote by $X_0,X_1$ the objects of the dg category $\mathscr{I}_0$, generated by a closed degree 0 morphism $f\in\mathscr{I}_0(X_0,X_1)$ (our former notations for these objects were 0 and 1). Then define $\mathscr{I}:=\mathscr{I}_0^\pretr$, and $\mathscr{D}:=\mathscr{I}/\mathscr{J}$ where $\mathscr{J}$ is the full dg subcategory with a single object $\Cone(f)$. Finally, consider the full dg subcategory $\mathscr{D}_0$ of $\mathscr{D}$, whose objects are $X_0$ and $X_1$. Lemma \ref{lemmadrk} states that $\mathscr{D}_0$ is isomorphic to $\mathscr{K}$, the Kontsevich dg category, introduced in Section \ref{sectionijwu}.

\sevafigc{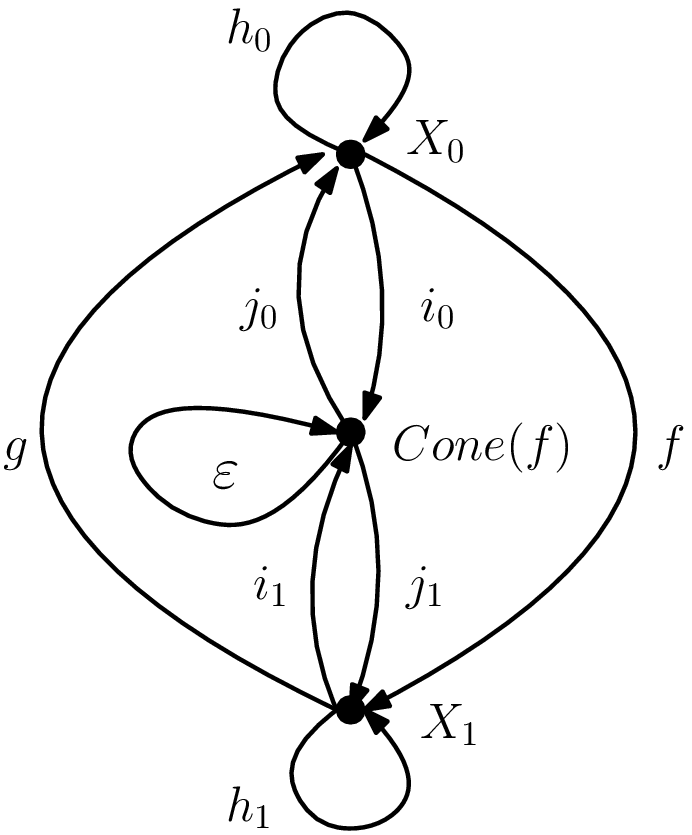}{60mm}{0}{The derivation of the Kontsevich dg category $\mathscr{K}$ from the Drinfeld dg quotient.}

To describe $\mathscr{D}_0$ explicitly, consider the fragment of the dg category $\mathscr{D}$, drawn in Figure 1. 

We start with the morphism $f$ of degree 0, $df=0$.

Then there are morphisms (in notations of Figure 1):
\begin{itemize}
\item $i_1$ of degree 1, $j_0$ of degree 0, 
\item $i_0$ of degree 0,  $j_1$ of degree -1, 
\item $\varepsilon$ of degree -1 (it is the morphism which was added in passage to the Drinfeld dg quotient).
\end{itemize}
One has:
\begin{equation}\label{eqa1}
j_0i_0=\id,\ \ j_1i_1=\id, \ \  j_1i_0=0, \ \ j_0i_1=0,\ \  i_0j_0+ i_1j_1=\id_{\Cone(f)}
\end{equation}
and
\begin{equation}\label{eqa2}
di_1=0,\ \ dj_0=0,\ \ di_0=i_1f,\ \ dj_1=fj_0, \ \ d\varepsilon=\id_{\Cone(f)}
\end{equation}
On the basis of these morphisms we define
\begin{equation}\label{eqa3}
g:=  j_0\varepsilon i_1, \ \ h_0:=j_0\varepsilon i_0,\ \  h_1:=j_1\varepsilon i_1, \ \ r:=j_1\varepsilon i_0
\end{equation}
One checks directly from \eqref{eqa1} and \eqref{eqa2} that the relations \eqref{kcat1},\eqref{kcat2} hold for these morphisms. 
One can show that the full dg subcategory $\mathscr{D}_0$ of $\mathscr{D}$, whose objects are $X_0$ and $X_1$, is generated by $f,g,h_0,h_1,r$, and the relations as above. 

It identifies the Kontsevich dg category $\mathscr{K}$ with a full subcategory in the dg quotient. Then, the standard results such as [Dr, 3.4] are applied to compute the cohomology of all Hom complexes in $\mathscr{K}$. See [Dr, 3.7.2-3.7.4].

\comment

\section{\sc Acyclicity of the bar complex of a weakly unital dg algebra}
This Appendix is written in aim to make more reader-friendly the formulas for homotopies used in the proof of Theorem \ref{theoremcomp1}, see \eqref{homotopy1}, \eqref{homotopy2}, \eqref{homotopy3}. Here we show how a similar argument works in the simplest possible situation. 

Let $A$ be an associative dg algebra with (strict) unit $j$, $\Bar(A)$ its bar-complex. It is well-known that $\Bar(A)$ is acyclic, and it is proved by a very simple homotopy
$$
h(a_1\otimes a_2\otimes \dots \otimes a_n)=j\otimes a_1\otimes a_2\otimes\dots\otimes a_n
$$
Now assume that $A$ is a weakly unital dg algebra, in the sense of [KS], see Definition \ref{defwu}.
Define $\Bar(A)$ as in the classical case. Denote by $j$ the weak unit. 

\begin{lemma}
The complex $\Bar(A)$ is acyclic in all degrees. 
\end{lemma}
\begin{proof}
Define a contracting homotopy by
\begin{equation}\label{homotopy4}
\begin{aligned}
\ &h(a_1\otimes a_2\otimes\dots\otimes a_n)=p_1(1)\otimes a_1\otimes a_2\otimes\dots\otimes a_n\pm\\
&p_2(1,a_1)\otimes a_2\otimes\dots\otimes a_n\pm p_3(1,a_1,a_2)\otimes a_3\otimes\dots \otimes a_n\pm\dots\pm\\
&p_{n+1}(1\otimes a_1\otimes a_2\otimes\dots\otimes a_n)
\end{aligned}
\end{equation}
Recall that $p_1(1)=j$ is the weak unit. 

Recall that
$$
\begin{aligned}
\ &[d,p_{k+1}(1,a_1,\dots,a_k)]=p_k(1,a_1,\dots,a_{k-1})*a_k\pm\\
&p_k(1,a_1*a_2,a_3,\dots)\pm p_k(1,a_1,a_2*a_3,\dots)\pm\dots\pm p_k(1,a_1,a_2,\dots,a_{k-1}*a_k)
\end{aligned}
$$
(see \eqref{relop2}).

It is left to the reader to check that
$$
dh+hd=\id
$$
\end{proof}

\endcomment

\bigskip
{\noindent\sc P.P.:}

{\small
\noindent {\sc Universiteit Antwerpen, Campus Middelheim, Wiskunde en Informatica, Gebouw G\\
Middelheimlaan 1, 2020 Antwerpen, Belgi\"{e}}}

\bigskip

\noindent{{\it e-mail}: {\tt Piergiorgio.Panero@uantwerpen.be}}

\bigskip
{\noindent\sc B.Sh.:}

{\small
\noindent {\sc Universiteit Antwerpen, Campus Middelheim, Wiskunde en Informatica, Gebouw G\\
Middelheimlaan 1, 2020 Antwerpen, Belgi\"{e}}}
\bigskip

{\small
\noindent{\sc Laboratory of Algebraic Geometry,
National Research University Higher School of Economics,
Moscow, Russia}}

\bigskip

\noindent{{\it e-mail}: {\tt Boris.Shoikhet@uantwerpen.be}}


\begin{thebibliography}{999}
{\footnotesize

\bibitem[Ba]{Ba} M.A.Batanin, Monoidal globular categories as a natural environment for the theory of weak n-categories,  {\it Adv. Math.} {\bf 136}(1) (1998), 39-103

\bibitem[BCSW]{BCSW} R.Betti, A.Carboni, R.Street, R.Walters, Variation through enrichment, {\it Journal of Pure and Applied Algebra}, {\bf 29} (1983), 109-127

\bibitem[COS]{COS} A.Canonaco, M.Ornaghi, P.Stellari, Localizations of the category of $A_\infty$ categories and internal Homs, preprint arXiv 1811.07830

\bibitem[Dr]{Dr} V.Drinfeld, DG quotients of DG categories, {\it J. Algebra} {\bf 272}(2),  (2004), 643-691

\bibitem[DS]{DS} W. G. Dwyer, J. Spalinski, Homotopy theories and model categories, in: {\it Handbook on Algebraic Topology}, Elsevier, 1995

\bibitem[GJ]{GJ} P.G.Goerss, J.F.Jardine, {\it Simplicial Homotopy Theory}, Birkh\"{a}user Progress in Mathematics, Vol. 174, 1999


\bibitem[GS]{GS} P.Goerss and K.Schemmerhorn, Model Categories and Simplicial Methods, Notes from lectures given at the University of Chicago, August 2004, available at P.Goerss' webpage

\bibitem[Hi]{Hi} Ph.S.Hirschhorn, {\it Model Categories and Their Localizations}, AMS Mathematical Surveys and Monographs, Vol. 99, 2003

\bibitem[Ho]{Ho} M.Hovey, {\it Model categories}, AMS Math. Surveys and Monographs, vol. 62, 1999

\bibitem[K1]{K1} M.Kontsevich,  Lectures at ENS, Paris, Spring 1998: notes taken by J.Bellaiche, J.-F.Dat, I.Marin, G.Racinet and H.Randriambololona, unpublished

\bibitem[K2]{K2} M.Kontsevich, Homological algebra of Mirror Symmetry, {\it Proceedings of the International Congress of Mathematicians}, Z\"{u}rich 1994, vol. I, Birkhauser 1995, 120-139

\bibitem[KL]{KL} G.M.Kelly, S.Lack, $V$-Cat is locally presentable or locally bounded if $V$ is so, {\it Theory and Applications of categories}, vol. 8, no. 23, 2001, 555-575

\bibitem[KS]{KS} M.Kontsevich, Y.Soibelman, Notes on $A_\infty$ algebras, $A_\infty$ categories, and non-commutative geometry, in: "Homological Mirror Symmetry: New Deevelopments and Perspectives" (A.Kapustin et al. (Eds.)), Lect. Notes in Physics 757 (Springer, Berlin Heidelberg 2009) 153-219

\bibitem[Li]{Li} F.E.J. Linton, Coequalizers in categories of algebras, in: Springer Lecture Notes in Mathematics 80, 75-90

\bibitem[Ly]{Ly} V.Lyubashenko, Homotopy unital $A_\infty$ algebras, J. Algebra 329 (2011) no. 1, 190-212

\bibitem[LyMa]{LyMa} V.Lyubashenko, O.Manzyuk, Unital $A_\infty$ categories, Problems of topology and related questions (V.V. Sharko, ed.), vol. 3, Proc. of Inst. of Mathematics NASU, Inst. of Mathematics, Nat. Acad. Sci. Ukraine, Kyiv, 2006, n. 3, 235-268

\bibitem[ML]{ML} S.Mac Lane, {\it Categories for the working mathematician}, 2nd ed., Springer GTM, 1998

\bibitem[MLM]{MLM} S.Mac Lane, I.Moerdijk, {\it Sheaves in Geometry and Logic: a first introduction to Topos Theory}, Springer, 1992

\bibitem[R]{R} E. Riehl, {\it Category Theory in context}, Dover Publ., 2016

\bibitem[Tab]{Tab} G.Tabuada, A Quillen model structure on the category of dg categories
C. R. Math. Acad. Sci. Paris, 340 (2005), 15-19

\bibitem[Wo]{Wo} H.Wolff, $V$-cat and $V$-graph,  {\it Journal of Pure and Applied Algebra}, {\bf 4}(1974), 123-135

}
\end{thebibliography}
\end{document}